\documentclass[onefignum,onetabnum]{siamart190516}

\usepackage{lipsum}
\usepackage{amsfonts}
\usepackage{amssymb}
\usepackage{graphicx}
\usepackage{epstopdf}
\usepackage{algorithmic}
\usepackage{mathtools}
\usepackage{multirow}
\usepackage{mleftright}
\usepackage{caption}
\usepackage{booktabs}
\usepackage{makecell}
\usepackage{subcaption}
\usepackage{algorithmic} 

\Crefname{ALC@unique}{Line}{Lines} 
\ifpdf
  \DeclareGraphicsExtensions{.eps,.pdf,.png,.jpg}
\else
  \DeclareGraphicsExtensions{.eps}
\fi

\newsiamremark{remark}{Remark}
\newsiamremark{notation}{Notation}
\newsiamremark{hypothesis}{Hypothesis}
\crefname{hypothesis}{Hypothesis}{Hypotheses}
\newsiamthm{claim}{Claim}

\headers{Approximate Generalized Feedback Nash Equilibria}{Laine, Fridovich-Keil, Chiu, and Tomlin}

\title{The Computation of Approximate Generalized Feedback Nash Equilibria\thanks{This research is supported by the DARPA Assured Autonomy Program, the ONR Basic Research Challenge on Multibody Systems, the NASA ULI program, SRC CONIX, and the NSF CPS program.}}

\author{Forrest Laine \thanks{Electrical Engineering and Computer Sciences, University of California, Berkeley. (\email{forrest.laine@berkeley.edu}).} 
\and David Fridovich-Keil\thanks{Aerospace Engineering and Engineering Mechanics, The University of Texas at Austin} \and Chih-Yuan Chiu\footnotemark[2] \and Claire Tomlin\footnotemark[2] }

\usepackage{amsopn}
\usepackage{dirtytalk}

\newcommand{\xinit}{\hat x_1}

\newcommand{\revise}[1]%
    {#1}

\ifpdf
\hypersetup{
  pdftitle={GFNE},
  pdfauthor={Laine}
}
\fi

\allowdisplaybreaks

\begin{document}

\maketitle

\begin{abstract}
  We present the concept of a Generalized Feedback Nash Equilibrium (GFNE) in dynamic games, extending the Feedback Nash Equilibrium concept to games in which players are subject to state and input constraints. We formalize necessary and sufficient conditions for (local) GFNE solutions at the trajectory level, which enable the development of efficient numerical methods for their computation. Specifically, we propose a Newton-style method for finding game trajectories which satisfy necessary conditions for an equilibrium, which can then be checked against sufficiency conditions. We show that the evaluation of the necessary conditions in general requires computing a series of nested, implicitly-defined derivatives, which quickly becomes intractable. To this end, we introduce an approximation to the necessary conditions which is amenable to efficient evaluation, and in turn, computation of solutions. We term the solutions to the approximate necessary conditions Generalized Feedback Quasi-Nash Equilibria (GFQNE), and we introduce numerical methods for their computation. In particular, we develop a Sequential Linear-Quadratic Game approach, in which a LQ local approximation of the game is solved at each iteration. The development of this method relies on the ability to compute a GFNE to inequality- and equality-constrained LQ games, and therefore specific methods for the solution of these special cases are developed in detail. We demonstrate the effectiveness of the proposed solution approach on a dynamic game arising in an autonomous driving application. 
  \end{abstract}



\section{Introduction}
\label{sec:intro}
There has been a recent growing interest in the application of game-theoretic concepts to applications in automated systems, as explored in  \cite{laine2021multi,wang2018game,di2019newton,cleac2019algames,fridovich2019efficient,fridovich2019iterative,di2020local,fisac2018hierarchical}. Indeed, numerous problems arising in these domains can be modeled as games, and particularly as dynamic or repeated games. For the purposes considered in this paper, we restrict our attention to discrete-time dynamic games, in which players have the ability to influence the state of the game over a finite set of game \emph{stages}. Associated with dynamic/repeated game solutions are game trajectories, which capture the evolution of the \emph{continuous} game state and player inputs over the sequence of stages. 

Solutions to dynamic games have been studied extensively, as in \cite{basar1998dynamic}. Within the theory of dynamic games, there are a variety of information patterns that can be associated with the game, and each of these patterns results in a fundamentally different solution. Perhaps the simplest information pattern is the \emph{open-loop} pattern, in which the repeated nature of the game is ignored. In this pattern the stages of the game are combined into a single \emph{static} game, and the entire trajectory is chosen at once to satisfy a Nash, Stackelberg, or other type of equilibrium. When constraints on the state variables are imposed upon the players, a generalized equilibrium must be considered.  Formulating games with an open-loop information pattern has many advantages as the resultant static game often admits known methods for analysis and computing solutions, as presented, for example, in  \cite{facchinei2009generalized,facchinei2010generalized}. However, by ignoring the dynamic nature of these games, the expressiveness of the resultant solutions are significantly limited. Intelligent game play in repeated games often involves observing the evolving game state and reacting accordingly. 

Reactive game-play can emerge when associating a \emph{closed-loop} information pattern to the game of interest. Effectively, games with this information pattern are such that the players choose control \emph{policies} which define the control input as a function of the game state at that stage. When those policies at each game stage are chosen to constitute an equilibrium for the dynamic subgame played over the subsequent game stages, the resultant solution is called a feedback equilibrium. This type of solution is capable of capturing strategies for a player which anticipate and account for the reaction of other players. Some advantages of this type of solution are explored in the autonomous driving context in, e.g., \cite{laine2020multihypothesis}.

While feedback equilibrium solutions are often desirable over their open-loop counterparts, for all but simple cases, there do not exist well-developed numerical routines for computing them. The unconstrained Linear-Quadratic (LQ) setting is perhaps the simplest case, and methods for computing feedback equilibria for these games are well known, as presented in \cite{basar1998dynamic}. The extension to the computation of a feedback Nash equilibrium for a class of inequality-constrained LQ games is introduced in \cite{reddy2017feedback,reddy2019open}, although restrictive assumptions are made on the form of the dynamics, constraints and cost terms of the game. Numerous other approaches have considered the computation of feedback equilibria under various special cases, such as those in \cite{tanwani2018feedback,tanwani2019feedback,kossioris2008feedback}, among others. Methods for computing feedback Nash equilibria have been recently developed in the unconstrained, nonlinear case using a value-iteration based approach  \cite{herrera2019algorithm}, and an iterative LQ game approach  \cite{fridovich2019efficient}. Nevertheless, to the best of our knowledge, no methods exist for computing feedback equilibria in games with constraints appearing on both the state and input dimensions, both in the general LQ and nonlinear settings.  

Since many emerging applications of dynamic games involve nonlinear dynamics as well as nonlinear constraints on the game states and inputs, we have pursued the development of a robust and efficient method for computing feedback equilibria in this setting. The result of that work is the topic of this paper. 

The outline of the paper is the following. In \cref{sec:formulation} we introduce the concept of a Generalized Feedback Nash Equilibrium (GFNE), which formally defines the feedback concept in the constrained setting. We discuss pitfalls with a parameterized approach to encoding GFNE problems as a means to motivate and introduce a non-parametric alternative. We then develop necessary and sufficient conditions on game trajectories to satisfy a GFNE using this non-parametric formulation. Challenges associated with the computation of such GFNE are highlighted, and a close approximation is introduced which is amenable to efficient computation. Finally, numerical methods for the computation of such approximate solutions are developed in detail, for the equality-constrained LQ setting (\cref{sec:equality_lq}), inequality-constrained LQ setting (\cref{sec:inequality_lq}), and ultimately, the general nonlinear setting (\cref{sec:nonlinear_games}). We demonstrate our method on an application to autonomous driving in \cref{sec:example} and conclude the paper in \cref{sec:conclusion}.

\section{Formulation}
\label{sec:formulation}
We focus our attention to the class of $N$-player discrete-time, deterministic, infinite, general-sum dynamic games of discrete stage-length $T$. Let $\mathbf{N}$ denote the set $\{1,...,N\}$, and similarly $\mathbf{T}$ the set $\{1,...,T\}$. We also make use of the sets $\mathbf{T^+} := \mathbf{T} \cup \{T+1\}$,  $\mathbf{T_t} := \{t,...,T\}$, and $\mathbf{T_t^+} := \{t,...,T+1\}$.  The game state at each discrete time-step $t$ is represented by $x_t \in \mathcal{X} = \mathbb{R}^n$. The game is assumed to start at stage $t=1$, from a pre-specified initial state $\xinit$. Throughout this paper we refer to subgames starting at stage $t$, which refers to the game played over a portion of the original game, on the stages $\{t,...,T\}$. 

The evolution of the game state is described by the dynamic equation:
\begin{equation}
    \label{eq:dynamics}
    x_{t+1} = f_t(x_t, u_t^1, ..., u_t^N), \ \ t\in\mathbf{T},
\end{equation}
where $ u_t^i \in \mathcal{U}_t^i = \mathbb{R}^{m_t^i}$ are the control variables chosen by players $i\in\mathbf{N}$ at time $t$. Let $m_t:=\sum_{i=1}^N m_t^i$, and $m_t^{-i}:=m_t-m_t^i$. 


To simplify the notation in definitions and derivations, we make use of the following shorthand to refer to various sets of state and control variables:

\vspace{0.15em}
\emph{Notation reference:}
\begin{align*}
    x &:= \revise{\{}x_1,x_2,...,x_{T+1}\revise{\}}, \\
    u^i &:=\revise{\{}u_1^i,u_2^i,...,u_T^i\revise{\}}, \\
    u &:= \revise{\{}u^1,...,u^N\revise{\}}, \\
    u_t &:=\revise{\{}u_t^1,...,u_t^N\revise{\}}, \\
    u_t^{-i} &:= \revise{\{}u_t^1,...,u_t^{i-1},u_t^{i+1},...,u_t^N\revise{\}}, \\
    \revise{\{}u_t^i,u_t^{-i}\revise{\}} &:= \revise{\{}u_t^1,...,u_t^N\revise{\}} = u_t
\end{align*}
\revise{Note that the above collections may be interpreted as sets of variables or as column vectors. In particular, at times throughout this work we will utilize the interpretation that $u_t := [(u_t^1)^\intercal, ..., (u_t^N)^\intercal]^\intercal$}.  

Each player in the game is associated with time-separable cost-functionals:
\begin{equation}
    \label{eq:cost_functional}
    L^i(x,u^1,...,u^N) := \sum_{t=1}^T l_t^i(x_t,u_t) + l^i_{T+1}(x_{T+1})
\end{equation}


Furthermore, each player may have  stage-wise, non-dynamic, equality and inequality constraints.

\begin{subequations}
\begin{align}
    &0 = h_t^i(x_t,u_t), \ t\in\mathbf{T} &0 = h_{T+1}^i(x_{T+1}) \label{eq:equality_constraints} \\ 
    &0 \leq g_t^i(x_t,u_t) \ t\in\mathbf{T,} &0 \leq g_{T+1}^i(x_{T+1}) \label{eq:inequality_constraints}
\end{align}
\end{subequations}

Let the dimension of the constraints $h^i_t$ and $g^i_t$, for all $t\in\mathbf{T^+}$ and $i\in\mathbf{N}$, be denoted as $a^i_t \geq 0$ and $b^i_t \geq 0$, respectively. 
Define $V_t^i: \mathcal{X} \to \mathbb{R} \cup \{\infty\}$ as the \emph{Value-function} for player $i\in\mathbf{N}$ at stage $t\in\mathbf{T^+}$, and $Z_t^i: \mathcal{X} \times \mathcal{U}_t^i \times ... \times \mathcal{U}_t^N \to \mathbb{R} \cup \{\infty\}$ the \emph{Control-Value-function} for player $i\in\mathbf{N}$ at time $t\in\mathbf{T}$.  

A Generalized Feedback Nash Equilibrium is defined in terms of maps $\pi_t^i: \mathcal{X}\to\mathcal{U}_t^i$, for $i\in\mathbf{N}$, $t\in\mathbf{T}$, which we refer to as feedback policies or strategies. The feedback policies, Value-functions, and Control-Value-functions are together defined according to the following recursive relationships \cref{eq:value_function_terminal}-\cref{eq:value_functions}:

\begin{equation}
    \label{eq:value_function_terminal}
    V_{T+1}^i(x_{T+1}) := \left\{\begin{array}{@{}ll@{}}
        \multirow{2}{*}{$l_{T+1}^i(x_{T+1})$,} & 0 = h_{T+1}^i(x_{T+1}) \\
        & 0 \leq g_{T+1}^i(x_{T+1}) \\
        \infty, & \text{else}
        \end{array}\right. 
\end{equation}
Given $V_{t+1}^i$ for some $t \in \mathbf{T}$ and $i\in\mathbf{N}$, we define $Z_t^i$ by
\begin{equation}
    \label{eq:q_functions}
    Z_t^i(x_t,u_t^1,...,u_t^N) := \left\{\begin{array}{@{}ll@{}}
        \multirow{2}{*}{$l_{t}^i(x_{t},u_t) + V_{t+1}^i(f_t(x_t,u_t))$,} & 0 = h_{t}^i(x_t,u_t) \\
        & 0 \leq g_t^i(x_t,u_t) \\
        \infty, & \text{else}
        \end{array}\right. 
\end{equation}
For a particular state $x_t$ at stage $t$, the feedback policies $\pi_t$ are defined to return a local generalized Nash equilibrium solution for the static game defined in terms of the $N$ Control-Value-functions evaluated at $x_t$ (one for each player). 
\begin{equation}
    \label{eq:policies}
    \begin{alignedat}{5}
    \tilde{u}_t = \pi_t(x_t) &\Longrightarrow \\ Z_t^1(&x_t,\tilde{u}_t^{1},...,\tilde{u}_t^{N}) &&\leq Z_t^1(x_t,u_t^1,\tilde{u}_t^{2},...,\tilde{u}_t^{N}), 
     \ \ &&\forall u_t^1 \in \mathcal{N}(\tilde{u}_t^1), \\
     &&& \ \ \vdots \\
     Z_t^N(&x_t,\tilde{u}_t^{1},...,\tilde{u}_t^{N}) &&\leq Z_t^N(x_t,\tilde{u}_t^{1},...,\tilde{u}_t^{N-1},u_t^N), 
     \ \ \ \ &&\forall u_t^N \in \mathcal{N}(\tilde{u}_t^N),
    \end{alignedat}
\end{equation}
The set $\mathcal{N}(\tilde{u}_t^i)$ is some neighborhood around $\tilde{u}_t^i$. \revise{It is possible that no such generalized equilibrium exists, due to the non-convexity of the Control-Value-functions w.r.t. the control variables \cite{dutang2013existence}. In such a case, the feedback policies are undefined, as are all preceding Value and Control-Value functions.} If a generalized equilibrium exists, it may be non-unique, and potentially non-isolated. For the purposes considered here, we require only that for any state $x_t$, the policies evaluate to one arbitrarily chosen, yet particular, local equilibrium. A more stringent definition for the policies $\pi_t$ could require that the inequalities in \cref{eq:policies} hold over the entire sets $\mathcal{U}_t^i$ rather than some local neighborhood. In any case, the Value-functions for stages $t\in\mathbf{T}$ are defined as 
\begin{equation}
    \label{eq:value_functions}
    V_t^i(x_t) := Z_t^i(x_t,\pi_t^1(x_t),...,\pi_t^N(x_t)).
\end{equation} 

\revise{At each stage of the game, it may be that all player's individual constraints \cref{eq:equality_constraints,eq:inequality_constraints} are impossible to satisfy for arbitrary choice of the game state at that stage. If for some particular state $x_t$, player $i$'s constraints are inadmissible, the associated value $V_t^i(x_t)$ will be infinite. Without computing the feedback policies $\pi_s, \ s\in\mathbf{T_{t+1}}$, it is not possible to determine whether the constraints for any given player are admissible, since their satisfaction depends on the actions of all other players in the remaining game stages. Nevertheless, throughout this work, we assume that the constraints for all players are indeed admissible, when starting from known initial state $\hat{x}_1$, and with other players choosing their actions according to the equilibrium feedback policies.}

\begin{definition}[GFNE] A Local Generalized Feedback Nash Equilibrium originating from $\hat{x}_1$ is defined by a set of policies $\pi_t^i$, $t\in\mathbf{T}$, $i\in\mathbf{N}$ defined in \cref{eq:policies}, such that the value of $V_1^i(\hat{x}_1)$, defined in \cref{eq:value_functions}, is finite for all $i\in\mathbf{N}$. In this paper we refer to local GFNE whenever we write GFNE. 
\end{definition}

Consider a collection of policies constituting a GFNE. Let the corresponding \emph{equilibrium trajectory} be denoted by $x_t^*, u_t^*$ such that
\begin{equation}
\label{eq:gfne_traj}
\begin{alignedat}{3}
    x_1^* &:= \xinit, \\
    u_t^{i*} &:= \pi_t^i(x_t^*), \ &&t\in\mathbf{T}, \\
    x_{t+1}^* &:= f_t(x_t^*, u_t^{1*},...,u_t^{N*}), \ \ \ &&t\in\mathbf{T}.
\end{alignedat}
\end{equation}

Despite the notion of a GFNE being well-defined, the presented formulation does not lend itself to the computation of such equilibrium trajectories. The remainder of this work is devoted to an alternate encoding of the GFNE problem, which is amenable to the computation. We begin the presentation of this approach by expressing the policies at each stage $t$ in terms of the subgame starting at that stage.

\begin{theorem}
\label{thm:infeasible}
The policies defined in \cref{eq:policies} can equivalently be expressed in terms of the nested Generalized Nash Equilibrium Problems \cref{eq:gn_problem_t}.


\begin{subequations}
\label{eq:gn_problem_t}
\begin{alignat}{3}
    \pi_t^i(x_t) := \tilde{u}^i_t \in \underset{u_t^i }{\textrm{\emph{arg}}} \underset{\substack{u_{t:T}^i,\\
    \tilde{u}^{-i}_{t+1:T},\\
    x_{t+1:T+1}}}{\textrm{\emph{min}}}\ 
        &\sum_{s=t}^T l_s^i(x_s,u_s^i, \tilde{u}_s^{-i}) + l_{T+1}^i(x_{T+1}) 
    \\
    \textrm{\emph{s.t.}} \ \ \ \ 
     & 0 = \tilde{u}^{-i}_s - \pi^{-i}_s(x_s), \ &&s\in\mathbf{T_{t+1}} \label{eq:policy_of_other_controls}\\
    &0 = x_{s+1} - f_s(x_s,u_s^i, \tilde{u}_s^{-i}), \ &&s\in\mathbf{T_t} \label{eq:reg_dynamics} \\
    & 0 = h_s^i(x_s,u_s^i, \tilde{u}_s^{-i}), &&s\in\mathbf{T_t} \label{eq:reg_eq}\\
    & 0 \leq g_s^i(x_s,u_s^i, \tilde{u}_s^{-i}), &&s\in\mathbf{T_t} \label{eq:reg_ineq}\\
    & 0 = h^i_{T+1}(x_{T+1}) \label{eq:reg_ult_eq} \\
    & 0 \leq g^i_{T+1}(x_{T+1}) \label{eq:reg_ult_ineq} 
\end{alignat}
\end{subequations}

\end{theorem}
Here $\pi_t^{-i}(x_t) := (\pi_t^1(x_t),...,\pi_t^{i-1}(x_t),\pi_t^{i+1}(x_t),...,\pi_t^N(x_t))$, and the notation $\underset{a}{\text{arg}} \ \underset{a,b}{\text{min}}$ is used to indicate that the minimum is taken over $a$ and $b$, but only the value of $a$ at the minimum is returned. \revise{The shared game state $x_s, s\in\mathbf{T_{t+1}^+}$, and control variables for all other players, $\tilde{u}_s^{-i}, s\in\mathbf{T_{t+1}}$ appear as decision variables in player $i$'s decision problem. This is because the constraints \cref{eq:policy_of_other_controls,eq:reg_dynamics} render these variables \emph{implicit} variables, allowing them to be effectively shared by all players \cite[Def. 2]{kim2019solving}.} The value of the minimization appearing in \cref{eq:gn_problem_t} is considered infinite for any combination of optimization variables violating the constraints. Note also that the set $\mathbf{T}_{T+1}$ is empty, so for stage $t=T$, the constraint \cref{eq:policy_of_other_controls} vanishes. 

\begin{proof}

Starting with stage $T$, and substituting $x_{T+1}$ using the dynamics \cref{eq:reg_dynamics}, and moving constraints \cref{eq:reg_eq,eq:reg_ineq,eq:reg_ult_eq,eq:reg_ult_ineq} into the objective by means of infinite-valued indicator functions, observe that the objective of the minimization is equivalent to $Z_T^i$ \cref{eq:q_functions} as claimed. 

Now, for some other stage $t\in\mathbf{T}$, assuming $\pi_{t}^i(x_{t+1})$ can be expressed by \cref{eq:gn_problem_t}, it can be shown that $\pi_t^i(x_t)$ must also be equivalently expressed by \cref{eq:gn_problem_t}. \cref{eq:gn_problem_t} can be re-written as the following:

\begin{subequations}
\label{eq:gn_problem_t_nested}
\begin{alignat}{3}
    \pi_t^i(x_t) := \tilde{u}^i_t \in \underset{u_t^i }{\textrm{arg}}&\underset{\substack{u_{t}^i,\\ \tilde{u}_{t+1}^{-i},\\x_{t+1}}}{\textrm{min}} \Bigg\{ \ l_t^i(x_t,u_t^i,\tilde{u}_t^{-i}) +  &\underset{\substack{u_{t+1:T}^i,\\
    \tilde{u}^{-i}_{t+2:T},\\
    x_{t+2:T+1}}}{\textrm{min}}\ 
        &\mathrlap{\sum_{s=t+1}^T l_s^i(x_s,u_s^i, \tilde{u}_s^{-i}) + l_{T+1}^i(x_{T+1}) \Bigg\}} 
    \\
    &&\textrm{s.t.} \ \ \ 
      &0 = \tilde{u}^{-i}_s - \pi^{-i}_s(x_s),  &s\in\mathbf{T_{t+2}} \label{eq:maybe_ignore_policies} \\
    &&&0 = x_{s+1} - f_s(x_s,u_s^i, \tilde{u}_s^{-i}), \   &s\in\mathbf{T_{t+1}} \\
    &&& 0 = h_s^i(x_s,u_s^i, \tilde{u}_s^{-i}), &s\in\mathbf{T_{t+1}} \\
    &&& 0 \leq g_s^i(x_s,u_s^i, \tilde{u}_s^{-i}), &s\in\mathbf{T_{t+1}} \\
    &&& 0 = h^i_{T+1}(x_{T+1}) \\
    &&& 0 \leq g^i_{T+1}(x_{T+1}) \\
    &\textrm{s.t.} \ \ \ \mathrlap{0 = \tilde{u}_{t+1}^{-i} - \pi_{t+1}^{-i}(x_{t+1})} \label{eq:one_policy} \\
    &\ \ \ \ \ \ \ \mathrlap{0 = x_{t+1}-f(x_t,u_t^i,\tilde{u}_t^{-i})} \label{eq:one_dynamic} \\
    &\ \ \ \ \ \ \ \mathrlap{0 = h_t^i(x_t,u_t^i,\tilde{u}_t^{-i})}\label{eq:one_eq}\\
    &\ \ \ \ \ \ \ \mathrlap{0 \leq g_t^i(x_t,u_t^i,\tilde{u}_t^{-i})} \label{eq:one_ineq}
\end{alignat}
\end{subequations}

The nested minimum appearing in \cref{eq:gn_problem_t_nested} is exactly that appearing in \cref{eq:gn_problem_t} for stage $t+1$ (ignoring \cref{eq:maybe_ignore_policies} if $t+1=T$). Because the controls  $\tilde{u}^{-i}_{t+1}$ (for $t+1\leq T$) are constrained by the policies $\pi_{t+1}^{-i}(x_{t+1})$, the value of this nested minimization must 
equal the value function $V_{t+1}(x_{t+1})$ as defined in \cref{eq:value_functions} for any minimizer $u_{t+1}^i$, regardless of whether or not the minimizer corresponds to the particular one corresponding to $\pi_{t+1}^i(x_{t+1})$. By substituting $x_{t+1}$ using the constraint \cref{eq:one_dynamic}, and using infinite-valued indicator functions to move \cref{eq:one_eq,eq:one_ineq} into the objective of \cref{eq:gn_problem_t_nested}, we see that the objective of the minimization is equivalent to \cref{eq:q_functions} for stage $t$, despite $u^i_{t+1:T}$ not being constrained by equilibrium policies. Thus, the alternate definition of $\pi_t^i(x_t)$ in \cref{eq:gn_problem_t} is equivalent to that in \cref{eq:policies} for all stages $t$. 
\end{proof}

Here we have defined the GFNE policies in terms of the \emph{nested} equilibrium problems with equilibrium constraints \cref{eq:gn_problem_t}. These equilibrium constraints arise in these problems because the constraints \cref{eq:policy_of_other_controls} are defined in terms of equilibrium problems. Critically though, the set of players in all inner-level equilibrium problems are exactly those in the outer-level problems, allowing for the removal of the redundant constraint that $u_s^i=\pi_s^i(x_s)$, $s\in\mathbf{T_{t+1}}$ from player $i$'s problem statement \cref{eq:gn_problem_t}, as demonstrated in \cref{thm:infeasible}. When the necessary conditions of all players are concatenated, the constraints $u_s^i=\pi_s^i(x_s)$, $s\in\mathbf{T_{t+1}}$ become redundant for \emph{all} $i$, as we show in \cref{thm:foncs}.  This fact will allow for a compact representation of necessary conditions associated with solutions of a GFNE, and ultimately algorithms for finding such solutions.

\begin{theorem}[Necessary Conditions]
\label{thm:foncs}
For some stage $t \in\mathbf{T}$, consider any set of policies $\pi_s^i$, $s\in\mathbf{T_t}$, $i\in\mathbf{N}$, as defined in \cref{eq:gn_problem_t}. Let the state $\hat{x}_t$ be such that a solution exists to equilibrium problem \cref{eq:gn_problem_t} at stage $t$. Denote the resultant sub-game solution trajectory by $\{x_s^*; \ s\in\mathbf{T_t^+}, \ x_t^*=\hat{x}_t\}$, $\{u_s^*; \ s\in\mathbf{T_t}\}$. Furthermore, assume the policies $\pi_s(x_s)$ are differentiable at the point $x_s^*$ for $s\in\mathbf{T_{t+1}}$, and a standard constraint qualification such as the linear independence constraint qualification holds for the optimization problem appearing in \cref{eq:gn_problem_t}, for each $i\in\mathbf{N}$. 
Then there exist multipliers $\{\lambda_s^i\in\mathbb{R}^n; \ s\in\mathbf{T_t}, \ i\in\mathbf{N}\}$, $\{\mu_s^i\in\mathbb{R}^{a^i_s}; \ s\in\mathbf{T_t^+}, \ i\in\mathbf{N}\}$, $\{\gamma_s^i\in\mathbb{R}^{b^i_s}; \ s\in\mathbf{T_t^+}, \ i\in\mathbf{N}\}$, and $\{\psi_s^i\in\mathbb{R}^{m^{-i}};\ s\in\mathbf{T_{t+1}}, \ i\in\mathbf{N}\}$ which satisfy\footnote{The notation $[\cdot]_{u_t}$ implies that if $u_t$ appears in the $j_1$ through $j_2$ indices of $z_t$, then the $j_1$ through $j_2$ rows of the matrix argument are selected. The notation in the first equation of \cref{eq:foncs} is used to indicate that the gradients of the functions $l_s^i(x_s,u_s)$, $f(x_s,u_s)$, $h_s^i(x_s,u_s)$, and $g_s^i(x_s,u_s)$ are evaluated at $x_s^*$ and $u_s^*$. The symbol $\perp$ is used to indicate complementarity of the left- and right-hand-side conditions.  For example, if the $k$-th element of $g_s^i(x_s^*,u_s^*) > 0$ then the $k$-th element of $\gamma_s^i$ must be $0$, and vice-versa. As before, for the final stage $t=T$, the set of conditions \cref{eq:foncs_x_tp1,eq:foncs_u_tp1} is empty, as the set $\mathbf{T_{T+1}} = \emptyset$.} :

\begin{subequations}
    \label{eq:foncs}
    \begin{alignat}{3}
     0 &= \nabla_{u_s^i} \bigg[ l_s^{i} + f_s^\intercal \lambda_{s}^i - h_s^{i\intercal} \mu_s^i  -  g_s^{i^\intercal} \gamma_s^i    &\bigg]_{x_s^*,u_s^*}, \ \ &i\in\mathbf{N}, \ s \in \mathbf{T_t} \\
     0 &= \nabla_{x_s} \bigg[ l_s^i - \lambda_{s-1}^i + f_s^\intercal \lambda_{s}^i - h_s^{i^\intercal} \mu_s^i  -  g_s^{i\intercal} \gamma_s^i  + \pi_s^{-i\intercal} \psi_s^i &\bigg]_{x_s^*,u_s^*},   \ \ &i\in\mathbf{N}, \ s\in\mathbf{T_{t+1}}, \label{eq:foncs_x_tp1}\\
     0 &= \nabla_{u^{-i}_s} \bigg[ l_s^i + f_s^\intercal \lambda_{s}^i - h_s^{i^\intercal} \mu_s^i  -  g_s^{i\intercal} \gamma_s^i  - \psi_s^i &\bigg]_{x_s^*,u_s^*},   \ \ &i\in\mathbf{N}, \ s\in\mathbf{T_{t+1}}, \label{eq:foncs_u_tp1}\\
     0 &= \nabla_{x_{T+1}} \bigg[ l_{T+1}^i - \lambda_{T}^i - h_{T+1}^{i^\intercal} \mu_{T+1}^i  -  g_{T+1}^{i\intercal} \gamma_{T+1}^i &\bigg]_{x_{T+1}^*},  \ \ &i\in\mathbf{N}, \\
    0 &= x^*_{s+1} - f_s(x^*_s,u^*_s), &&s\in\mathbf{T_t}, \\
    0 &= h_s^i(x_s^*,u_s^*), &&i\in\mathbf{N}, \ s\in\mathbf{T_t},\\
    0 &\leq g_s^i(x_s^*,u_s^*) \perp \gamma_s^i \geq 0, &&i\in\mathbf{N}, \ s\in\mathbf{T_t}, \label{eq:foncs_ineq_t} \\
    0 &= h_{T+1}^i(x_{T+1}^*), &&i\in\mathbf{N},\\
    0 &\leq g_{T+1}^i(x_{T+1}^*) \perp \gamma_{T+1}^i \geq 0, &&i\in\mathbf{N}. \label{eq:foncs_ineq_T}
    \end{alignat}
\end{subequations}

\revise{When strict complementarity holds (w.r.t. \cref{eq:foncs_ineq_t,eq:foncs_ineq_T}), let $L_t(z_t,x_t^*)=0$ denote the entire set of conditions \cref{eq:foncs} formed by treating active inequalities \cref{eq:foncs_ineq_t,eq:foncs_ineq_T} as equalities, and ignoring the inactive inequalities.  Here, $z_t$ is the set of all variables appearing in \cref{eq:foncs} other than $x_t^*$ \emph{and} all multipliers $\gamma_s^i$ associated with inactive inequality constraints. For example, if at a solution to \cref{eq:foncs}, it is that $\gamma_{s,j}^i = 0$ for some $i\in\mathbf{N},s\in\mathbf{T_t},j\in\{1...b_s^i\}$, then $\gamma_{s,j}^i=0$ is included in the system of equations $L_t(z_t,x_t^*)=0$, where as the expression involving $g_s^i(x_s^*,u_s^*)$ is discarded.}  If the matrix $\nabla_{z_t}L_t$ is non-singular, then the policy $\pi_t(x_t)$ is also differentiable at the point $\hat{x}_t$, and $\nabla_x \pi_t(x_t) := -[[\nabla_{z_t}L_t]^{-1} \nabla_{x_t}L_t]_{u_t}$. 
\end{theorem}

\begin{proof}
By the assumption that the optimization problems \cref{eq:gn_problem_t} satisfy a standard constraint qualification, \revise{the first-order necessary conditions for optimality (Karush-Kuhn-Tucker conditions) imply that there must exist multipliers associated with each player $i\in\mathbf{N}$'s optimization problem at stage $t$, such that both the conditions in \cref{eq:foncs} (evaluated for the particular $i\in\mathbf{N}$) and the constraints $u_s^{-i*}=\pi_s^{-i}(x_s)$ hold \cite[Thm. 12.1]{nocedal2006numerical}.} Concatenating these conditions for all of the $N$ players gives rise to the conditions \cref{eq:foncs}, with the addition of the constraints \cref{eq:policy_of_other_controls} for each $i\in\mathbf{N}$. Since there must exist multipliers satisfying those conditions, the conditions \cref{eq:foncs} must also be satisfied, as \cref{eq:foncs} are formed by simply removing the constraints $u_s^*=\pi_s(x_s)$. 

After removing all inactive inequality constraints and associated multipliers from \cref{eq:foncs}, it is straightforward to verify that the dimension of the system $L_t(z_t,x_t^*)$ and $z_t$ are the same, and therefore that $\nabla_{z_t}L_t$ is a square matrix. If $z_t^*$ is comprised of $x_{t+1:T+1}^*$, $u_{t:T}^*$ and multipliers satisfying \cref{eq:foncs} such that $L_t(z_t^*,x_t^*)=0$, then assuming this matrix is non-singular, the Implicit Function Theorem states that there must exist a unique function $\Pi_t(x_t)$ and open set $\mathcal{X}_t^*\subset \mathcal{X}$ containing $x_t^*$ such that $L_t(\Pi_t(x_t),x_t))=0$ for all $x_t\in\mathcal{X}_t^*$. By the uniqueness of this function, we must have that for all $x_t\in\mathcal{X}_t^*$, $[\Pi_t(x_t)]_{u_t} = \pi_t(x_t)$, where $[\Pi_t(x_t)]_{u_t}$ selects the components of the function $\Pi_t$ corresponding to the subset of $z_t$ containing $u_t$. Therefore, for all $x_t\in\mathcal{X}_t^*$, $\nabla_{x_t}\pi_t(x_t) = \nabla_{x_t}([\Pi_t(x_t)]_{u_t}) = -[[\nabla_{z_t}L_t]^{-1} \nabla_{x_t}L_t]_{u_t}$.
\end{proof}

Observe that the conditions $L_t(z_t,x_t)=0$ contain as a subset the conditions $L_{s}(z_{s},x_{s})=0$, $s\in\mathbf{T_{t+1}}$. If the matrices $\nabla_{x_s} L_s(z_s,x_s)$, $s\in\mathbf{T_{t+1}}$ are also non-singular, then in some neighborhood of $z_t^*$, the constraints $u_s^*=\pi_s(x_s^*)$, $s\in\mathbf{T_{t+1}}$ are equivalent to $L_s(z^*_s,x^*_s)=0$, motivating the removal of the constraints \cref{eq:policy_of_other_controls} from the conditions \cref{eq:foncs}. 

For games and corresponding solutions satisfying the stated assumptions, \cref{thm:foncs} suggests a method for computing those solutions. Evaluating the conditions \cref{eq:foncs} only requires the evaluation of the policy gradients $\nabla_{x_s}\pi_s(x_s)$, and not the policies themselves. 

The procedure we propose for computing GFNE trajectories is to find a solution to the conditions \cref{eq:foncs}, which can then be checked against a sufficiency condition to ensure that the solution indeed constitutes a GFNE. 

\begin{theorem}[Sufficient Conditions]
    \label{thm:sufficiency}
    Consider any set of states $\{x_s^*$, $s\in\mathbf{T^+_t}\}$ and controls $\{u_s^*$, $s\in\mathbf{T_t}\}$, which together with multipliers  $\{\lambda_s^i\in\mathbb{R}^n; \ s\in\mathbf{T_t}, \ i\in\mathbf{N}\}$, $\{\mu_s^i\in\mathbb{R}^{a^i_s}; \ s\in\mathbf{T_t^+}, \ i\in\mathbf{N}\}$, $\{\gamma_s^i\in\mathbb{R}^{b^i_s}; \ s\in\mathbf{T_t^+}, \ i\in\mathbf{N}\}$, and $\{\psi_s^i\in\mathbb{R}^{m^{-i}};\ s\in\mathbf{T_{t+1}}, \ i\in\mathbf{N}\}$ satisfy the conditions \cref{eq:foncs}, with the matrix $\nabla_{z_s}L_s$ non-singular for all $s\in\mathbf{T}$. If additionally, for all $i\in\mathbf{N}$,
    \begin{subequations}\label{eq:second_order}
    \begin{alignat}{4}
        & \ \  \mathrlap{\sum_{s=t}^T \begin{bmatrix} d_{x,s} \\ d_{u,s} \end{bmatrix}^\intercal \nabla_{x_s,u_s}^2 \mathcal{L}_s^{i*} \begin{bmatrix} d_{x,s} \\ d_{u,s} \end{bmatrix} + d_{x,T+1}^\intercal \nabla^2 l_{T+1}^i(x_{T+1}^*) d_{x,T+1} > 0,}  \\
        & \ \ \forall \{d_{x,s}, d_{u,s}\}: \ && 0 = d_{u_t}^{-i} - \nabla_x \pi_t^{-i}(x_t^*)d_{x,t}, \label{eq:suf_a} \\
        &&& 0 = d_{u_s} - \nabla_x \pi_s(x_s^*)d_{x,s},  &&s\in\mathbf{T_t}, \label{eq:suf_b} \\
        &&& 0 = d_{x,s+1} - \nabla_x f_s(x_s^*,u_s^*) d_{x,s} - \nabla_u f(x_s^*,u_s^*) d_{u_s}, \ \ &&s\in\mathbf{T_t} \label{eq:suf_c}
    \end{alignat}
    \end{subequations}
    then the trajectory $x_s^*$, $u_s^*$ constitutes a GFNE trajectory \revise{for the subgame originating from $x_t^*$ at stage $t$. In \cref{eq:second_order}, $\mathcal{L}_s^*$ is the stage-wise Lagrangian for player $i$, defined as the following:}
    \begin{equation}
    \revise{
        \mathcal{L}_s^* := \big[l_s^i - (\lambda_s^*)^\intercal (x_{s+1}^*-f_s) - (\mu_s^i)^\intercal h_s^i - (\gamma_s^i)^\intercal g_s^i - (\psi_s^i)^\intercal (u_s^* - \pi_s)\big]_{x_s^*,u_s^*}.}
    \end{equation}
\end{theorem}
\begin{proof}
The set of $d_{x,s}$, $d_{u,s}$ satisfying \cref{eq:suf_a,eq:suf_b,eq:suf_c} is a super-set of the critical constraint cone \cite[Eq. 12.53]{nocedal2006numerical}, therefore the trajectory $x_s^*,u_s^*$ must constitute a true local minimum of the problems \cref{eq:gn_problem_t} for stage $t$ and each player $i\in\mathbf{N}$, as stated in \cite[Theorem 12.6]{nocedal2006numerical}. 
\end{proof}

The sufficiency condition outlined in \cref{thm:sufficiency} is stricter than necessary since it ignores other active constraints which reduce the volume of the critical constraint cone, and could be relaxed by considering the linearization of all active constraints.

Although \cref{thm:foncs,thm:sufficiency} together outline a procedure for computing GFNE trajectories, there remain some difficulties which must be addressed if such a procedure is to be practical. While the policies $\pi_s(x_s), s\in\mathbf{T_{t+1}}$ do not appear in the conditions \cref{eq:foncs}, the policy gradients $\nabla_x \pi_s(x_s)$ do, and their evaluation requires solving a system of equations which generally depend on second-order derivatives of all policies $\pi_r(x_r)$, $r\in\mathbf{T_{s+1}}$. Furthermore, evaluating second-order derivatives of any policy requires evaluating third-order derivatives of all subsequent policies, and so forth. The effect of this is that $T^{th}$-order derivatives of dynamic and constraint functions appearing at the late stages of the game must be computed to evaluate the conditions \cref{eq:foncs} when $t=1$. While technically possible, this requirement is impractical for many games. \revise{We therefore introduce a reasonable approximation to the equilibrium policies $\pi_t(\cdot)$, which we denote by $\hat{\pi}_t(\cdot)$. By ignoring all higher-order derivatives of later-game policies $\hat{\pi}_s(\cdot)$, $s\in\mathbf{T_{t+1}}$, these approximate policies are computed recursively in a computationally tractable manner. }



\begin{definition}[Quasi-Equilibrium Policies]
\label{def:approx_policies}
\revise{We approximate $\pi_t(x_t)$ by  $\hat{\pi}_t(x_t)$, which for each $t\in\mathbf{T}$ is defined as producing controls $u_t^*$, such that there exist a sub-game trajectory and corresponding multipliers satisfying the following conditions:}




\begin{equation}
    \label{eq:qfoncs}
    \begin{alignedat}{3}
     0 &= \nabla_{u_s^i} \bigg[ l_s^{i} + f_s^\intercal \lambda_{s}^i - h_s^{i\intercal} \mu_s^i  -  g_s^{i^\intercal} \gamma_s^i    &\bigg]_{x_s^*,u_s^*}, \ \ &i\in\mathbf{N}, \ s \in \mathbf{T_t} \\
     0 &= \nabla_{x_s} \bigg[ l_s^i - \lambda_{s-1}^i + f_s^\intercal \lambda_{s}^i - h_s^{i^\intercal} \mu_s^i  -  g_s^{i\intercal} \gamma_s^i \bigg]_{x_s^*,u_s^*} +&\tilde{\nabla}_{x_s}\hat{\pi}_s^{-i\intercal} \psi_s^,   \ \ &i\in\mathbf{N}, \ s\in\mathbf{T_{t+1}}, \\
     0 &= \nabla_{u^{-i}_s} \bigg[ l_s^i + f_s^\intercal \lambda_{s}^i - h_s^{i^\intercal} \mu_s^i  -  g_s^{i\intercal} \gamma_s^i  - \psi_s^i &\bigg]_{x_s^*,u_s^*},   \ \ &i\in\mathbf{N}, \ s\in\mathbf{T_{t+1}}, \\
     0 &= \nabla_{x_{T+1}} \bigg[ l_{T+1}^i - \lambda_{T}^i - h_{T+1}^{i^\intercal} \mu_{T+1}^i  -  g_{T+1}^{i\intercal} \gamma_{T+1}^i &\bigg]_{x_{T+1}^*},  \ \ &i\in\mathbf{N}, \\
    0 &= x^*_{s+1} - f_s(x^*_s,u^*_s), &&s\in\mathbf{T_t}, \\
    0 &= h_s^i(x_s^*,u_s^*), &&i\in\mathbf{N}, \ s\in\mathbf{T_t},\\
    0 &\leq g_s^i(x_s^*,u_s^*) \perp \gamma_s^i \geq 0, &&i\in\mathbf{N}, \ s\in\mathbf{T_t}, \\
    0 &= h_{T+1}^i(x_{T+1}^*), &&i\in\mathbf{N},\\
    0 &\leq g_{T+1}^i(x_{T+1}^*) \perp \gamma_{T+1}^i \geq 0, &&i\in\mathbf{N}.
    \end{alignedat}
\end{equation}

The conditions \cref{eq:qfoncs} are nearly the same as \cref{eq:foncs}, \revise{although the policies $\pi_s(x_s)$ are replaced with the approximate policies $\hat{\pi}_s(x_s)$, and the gradients taken with respect to these approximate policies are replaced with what we denote as \emph{quasi-gradients}, $\tilde{\nabla}_{x_s} \hat{\pi}_s(x_s)$.} Letting $\hat{L}_t(z_t,x_t^*)$ represent the concatenation of active conditions in \cref{eq:qfoncs} (analogous to $L_t(z_t,x_t^*)$ in \cref{thm:foncs}), then the quasi-gradient of $\hat{\pi}_t(x_t)$ is defined as 
\begin{equation} \label{eq:ift_error}
    \tilde{\nabla}_{x_t}\hat{\pi}_t(x_t) := -[[\check{\nabla}_{z_t}\hat{L}_t]^{-1}\nabla_{x_t}\hat{L}_t]_{u_t},
\end{equation} \revise{where $\check{\nabla}_{z_t}\hat{L}_t$, is approximate the Jacobian of $\hat{L}_t$ formed by replacing any term of the form $\nabla_{x_s}[ \tilde{\nabla}_{x_s}\hat{\pi}_s(x_s)^{-i\intercal}\psi_s^i]$ (appearing in the true Jacobian $\nabla_{z_t}\hat{L}_t$) with the $0$ matrix.} If the matrix $\check{\nabla}_{z_t}\hat{L}_t$ is singular at some $(z_t,x_t)$, we say that the policy quasi-gradient does not exist at that point. 
\end{definition}

\revise{As with the definition of the true equilibrium policies \cref{eq:policies}, the policies $\hat{\pi}_t(x_t)$ are not true functions, in that there may be multiple solutions to \cref{eq:qfoncs}. Nevertheless, we treat these policies as mapping to one arbitrarily chosen yet particular solution. So long as $\check{\nabla}_{z_t}\hat{L}_t$ is non-singular, the solutions to \cref{eq:qfoncs} are isolated, and a quasi-gradient can be formed around any such solution. This resultant quasi-gradient can then be used to define the conditions \cref{eq:qfoncs} at the preceding stage $t-1$.}

The approximate policies $\hat{\pi}_t(x_t)$ and corresponding quasi-gradients can be evaluated without the need for computing any third- or higher-order derivatives of any constraint or objective terms of the game. Solutions satisfying the conditions \cref{eq:qfoncs} will not satisfy the conditions \cref{eq:foncs} in general. Rather, the solutions to \cref{eq:qfoncs} will be distinct from solutions to \cref{eq:foncs}, and therefore we introduce the notion of a Generalized Feedback Quasi-Nash Equilibrium (GFQNE) to characterize these solutions. 

\begin{definition}[GFQNE]
Let $\{x_s^*; \ s\in\mathbf{T_t^+}, \ x_t^*=\hat{x}_t\}$, $\{u_s^*; \ s\in\mathbf{T_t}\}$, $\{\lambda_s^i\in\mathbb{R}^n; \ s\in\mathbf{T_t}, \ i\in\mathbf{N}\}$, $\{\mu_s^i\in\mathbb{R}^{a^i_s}; \ s\in\mathbf{T_t^+}, \ i\in\mathbf{N}\}$, $\{\gamma_s^i\in\mathbb{R}^{b^i_s}; \ s\in\mathbf{T_t^+}, \ i\in\mathbf{N}\}$, and $\{\psi_s^i\in\mathbb{R}^{m^{-i}};\ s\in\mathbf{T_{t+1}}, \ i\in\mathbf{N}\}$ be such that the conditions \cref{eq:qfoncs} are satisfied. 
 Then we say that the trajectory $\{x_s^*; \ s\in\mathbf{T_t^+}, \ x_t^* = \hat{x}_t\}, \{u_s^*; \ s\in\mathbf{T_t}\}$ constitutes a Generalized Feedback Quasi-Nash Equilibrium (GFQNE) trajectory for the sub-game originating from $\hat{x}_t$ at stage $t$. 
\end{definition}

\revise{If all cost functionals \cref{eq:cost_functional} in the game are convex, and all dynamic \cref{eq:dynamics} and non-dynamic \crefrange{eq:equality_constraints}{eq:inequality_constraints} constraints are linear, then the true equilibrium policies \cref{eq:gn_problem_t} are piecewise linear, implying that $\tilde{\nabla}_{x_s}\hat{\pi}_s(x_s) \equiv \nabla_{x_s}\hat{\pi}_s(x_s)$, and $\check{\nabla}_{z_t}\hat{L}_t \equiv \nabla_{z_t}\hat{L}_t$. Therefore \cref{eq:qfoncs} are equivalent to \cref{eq:foncs}, and by the convexity of the cost functionals, any solution to \cref{eq:qfoncs} also satisfies \cref{eq:second_order}, implying that all GFQNE are GFNE in the convex-linear setting. }

\revise{In the general setting, in addition to not requiring GFQNE solutions to satisfy any second-order conditions analogous to \cref{eq:second_order}, the policy quasi-gradients will in general not equal the policy gradients. The difference between the GFQNE and GFNE trajectories can be understood in terms of two approximations introduced at each stage of the game, the effects of which are propagated over the game horizon $T$.}

\revise{The first approximation is the difference between the true gradient of a policy (either $\nabla \hat{\pi}_t$ or $\nabla \pi_t$) and the quasi-gradient which can be computed ($\tilde{\nabla}\hat{\pi}_t$ and $\tilde{\nabla}\pi_t$, respectively). Notice that for the ultimate stage of the game $t=T$, there are no sub-game policies to consider, and therefore the the set of possible equilibrium policies $\pi_T$ is a subset of the set of possible approximate equilibrium policies $\hat{\pi}_T$.  Furthermore, whenever $\pi_T \equiv \hat{\pi}_T$, it follows that $\tilde{\nabla} \hat{\pi}_T \equiv \nabla\hat{\pi}_T$.}

\revise{
For the penultimate stage $t=T-1$, the necessary conditions \cref{eq:qfoncs} are equivalent to \cref{eq:foncs}, which again implies that the set of possible equilibrium policies $\pi_{T-1}$ is a subset of the set of possible approximate equilibrium policies $\hat{\pi}_{T-1}$. However, even when $\pi_{T-1} \equiv \hat{\pi}_{T-1}$, $\tilde{\nabla} \hat{\pi}_{T-1} \neq \nabla\hat{\pi}_{T-1} \equiv \nabla\pi_{T-1}$. This is true not only for the penultimate stage, but for all stages $t<T-1$.  The error between the true and approximate policy gradients can be quantified by observing that, for any square matrices $A, B$ of compatible dimension such that $A$, $A+B$ are non-singular $\cite{henderson1981deriving}$:
\begin{equation}
    (A+B)^{-1} = A^{-1} - A^{-1}B(A+B)^{-1}
\end{equation}
where, in the context of \cref{eq:ift_error}, $(A+B)$ and $A$ correspond to ${\nabla}_{z_t}\hat{L}_t$ and $\check{\nabla}_{z_t}\hat{L}_t$, respectively.\footnote{\revise{The usual interpretation of $A+B$ in this setting would be for the perturbed or approximate matrix, although the reversed interpretation we assign is also valid.}} In this interpretation, $B$ is a matrix comprised of terms of the form $\nabla_{x_s}[\tilde{\nabla}_{x_s}\hat{\pi}_s(x_s)^{-i\intercal}\psi_s^i]$ appearing in the locations they do in $\nabla_{z_t}\hat{L}_t$, and zero elsewhere. By direct substitution, right-multiplying both sides by $\nabla_{x_t}\hat{L}_t$, bounding the norm of ``$B$,'' and applying the triangle inequality, it follows that 
\begin{equation}\label{eq:quasigraderror}
    \frac{\| \tilde{\nabla}_{x_t}\hat{\pi}_t(x_t) - \nabla_{x_t}\hat{\pi}_t(x_t) \|_2}{\|{\nabla}_{x_t}\hat{\pi}_t(x_t)\|_2} \leq  \epsilon \cdot N(T-t) \cdot \|\check{\nabla}_{z_t}\hat{L}_t^{-1}\|_2  ,
\end{equation}
where 
\begin{equation}
    \epsilon := \max_{s\in\mathbf{T_{t+1}},i\in\mathbf{N}} \| \nabla_{x_s}[\tilde{\nabla}_{x_s}\hat{\pi}_s(x_s)^{-i\intercal}\psi_s^i] \|_2.
\end{equation}
Analogous results for other choices of matrix norms are readily made. The inequality \cref{eq:quasigraderror} bounds the relative error of the policy gradient in terms of the second-order derivative terms of sub-game policies, and the approximate Jacobian $\check{\nabla}\hat{L}_t$. For highly non-linear policies $\hat{\pi}_s$, $\epsilon$ may be large, and the bound \cref{eq:quasigraderror} may be uninformative. 
}

\revise{
The second source of approximation first enters at stage $t=T-2$, and is reintroduced at every preceding stage. Given sub-game policies $\hat{\pi}_s,\ s\in\mathbf{T_{t+1}}$, each $\hat{\pi}_t$ would ideally evaluate to an equilibrium solution to problems analogous to \cref{eq:gn_problem_t}, though with the sub-game policies appearing in \cref{eq:policy_of_other_controls} replaced with the approximate sub-game policies $\hat{\pi}_s(x_s)$. We denote such policies as the idealized approximate policies. Instead, $\hat{\pi}_t$ is characterized as a solution to the approximate necessary conditions \cref{eq:qfoncs}, where all gradients of $\hat{\pi}_s$ are replaced with quasi-gradients. In general, the approximation error in $\tilde{\nabla}\hat{\pi}_{s}, s\in\mathbf{T_{t+1}}$ can greatly influence the resultant difference between solutions generated by $\hat{\pi}_t$ and the idealized policies, making it difficult (or potentially impossible) to bound the difference for non-infinitesimal errors.}

\revise{It is important to emphasize that these two forms of approximation are recursively made at each stage of the game. The difference between the true and approximate equilibrium policies will, in general, become increasingly pronounced when computing backwards from the ultimate stage $T$. Therefore it is likely that the GFNE and GFQNE trajectories are considerably different. Despite this,  it is simply intractable to compute GFNE trajectories for non-trivial games, and we believe GFQNE trajectories are the closest computationally tractable approximation possible. 
}


\subsubsection*{Note on constraint dimension}
So far we have also made an important, limiting assumption, which is that the matrices $\nabla_{z_t}L_t$ and $\check{\nabla}_{z_t}\hat{L}_t$ are non-singular for all $t\in\mathbf{T}$. For many common forms of constraints \cref{eq:equality_constraints}, \cref{eq:inequality_constraints}, this assumption cannot hold. This will occur, for example, when there is a terminal constraint on the entire game state, such as $h_{T+1}^i(x_{T+1}) := x_{T+1}$. If $m_t^i < n$, then the matrices $\check{\nabla}_{z_t}\hat{L}_t$ necessarily must be singular. Since many games involve constraints of this form, we handle them in the following way. 

If at any stage $t$, the matrix $\check{\nabla}_{z_t}\hat{L}_t$ is found to be singular, but the game is otherwise well-posed,\footnote{For example, the quadratic cost functionals of every player have sufficient curvature in the tangent cone of the game.} then this is likely due to an over-constrained sub-game. In this situation, we can combine the stage $t$ with the preceding stage $t-1$, and define new combined-stage dynamics and constraint functions accordingly. For example, assume at stage $t$ the matrix $\check{\nabla}_{z_t}\hat{L}_t$ is singular. We then define $\hat{u}_{t-1} := [u_{t-1}^\intercal \ u_t^\intercal]^\intercal$, $\hat{\mathcal{U}}_{t-1}^i := \mathcal{U}_{t-1}^i\times\mathcal{U}_t^i$, and the updated dynamic, constraint, and stage-wise cost functionals as
\begin{subequations}
\begin{align}
    \hat{f}_{t-1}(x_{t-1},\hat{u}_{t-1}) &:= f_t(f_{t-1}(x_{t-1},u_{t-1}),u_t), \\
    \hat{g}^i_{t-1}(x_{t-1},\hat{u}_{t-1}) &:= [g^i_{t-1}(x_{t-1},u_{t-1})^\intercal \ g^i_{t}(f(x_{t-1},u_{t-1}),u_t)^\intercal]^\intercal, \\
    \hat{h}^i_{t-1}(x_{t-1},\hat{u}_{t-1}) &:= [h^i_{t-1}(x_{t-1},u_{t-1})^\intercal \ h^i_{t}(f(x_{t-1},u_{t-1}),u_t)^\intercal]^\intercal, \\
    \hat{l}_{t-1}^i(x_{t-1},\hat{u}_{t-1}) &:= l_{t-1}^i(x_{t-1},u_{t-1}) + l_t^i(f(x_{t-1},u_{t-1}),u_t).
\end{align}
\end{subequations}  In this procedure, we effectively reduce the number of stages of the game by $1$, but the dimension of all controls input to the game and the cost and constraints imposed upon each player are unchanged. 

Throughout the remainder of this paper, we will assume that game stages are combined as necessary to ensure the subgame policy quasi-gradients are well-defined, and the game horizon $T$, dynamics, constraints, and cost-functionals all reflect any such modifications. In what follows we focus on the derivation of numerical methods for computing Generalized Feedback Quasi-Nash Equilibria. We begin our presentation by considering a special-case, which will serve as a building block for more general methods.

\section{Equality-Constrained LQ Games}
\label{sec:equality_lq}
We consider the case in which the dynamics equation describing the game evolution \cref{eq:dynamics} is linear in its arguments, the  cost-functionals \cref{eq:cost_functional} for each player are quadratic functions of the state and control variables, and each player is subject only to linear equality constraints.

In particular, let 

\begin{equation}
    \label{eq:linear_dynamics}
    x_{t+1} = A_t x_t + B_t^1 u_t^1 + ... + B_t^N u_t^N + c_t, \ \ t\in\mathbf{T},
\end{equation}
for (time-varying) matrices $A_t\in\mathbb{R}^{n\times n}$, $B_t^i \in \mathbb{R}^{m_t^i\times n}$, and vectors $c_t\in\mathbb{R}^n$. Associated with the dynamic constraints are multipliers $\lambda_t^i$ for each player $i\in\mathbf{N}$. Let 
\begin{equation}
\begin{aligned}
B_t &:= \begin{bmatrix} B_t^1&...&B_t^N\end{bmatrix}, \ \ \ \ \hat{B}_t := \begin{bmatrix} B_t^1 \\ & \ddots \\ & & B_t^N \end{bmatrix}, \ \ \ \ \lambda_t := \begin{bmatrix} \lambda_t^1 \\ \vdots \\ \lambda_t^N
\end{bmatrix}, \\
\tilde{B}_t &:= \begin{bmatrix} B_t^2 & \hdots & B_t^N \\
& & & \ddots \\
& & & & B_t^1 & \hdots & B_t^{N-1}
\end{bmatrix}.
\end{aligned}
\end{equation}

In this setting, the cost functionals for each player can be expressed as:
\begin{equation}
    \label{eq:quadratic_cost_functional}
    \begin{alignedat}{1}
    l^i_t(x_t,u_t) &:= \frac{1}{2}\bigg( \begin{bmatrix} x_t\\u_t\end{bmatrix}^\intercal \begin{bmatrix} Q_t^i & S_t^{i\intercal} \\ S_t^i & R_t^i\end{bmatrix} \begin{bmatrix} x_t\\u_t\end{bmatrix} + 2\begin{bmatrix} x_t\\u_t\end{bmatrix}^\intercal \begin{bmatrix} q_t^i \\ r_t^i \end{bmatrix} \bigg),  \\
    l^i_{T+1}(x_{T+1}) &:=\frac{1}{2} \bigg(x_{T+1}^\intercal Q_{T+1}^i x_{T+1} + 2x_{T+1}^\intercal q_{T+1}^i\bigg),
    \end{alignedat}
\end{equation}
for (time-varying) matrices $Q_t^i \in \mathbb{R}^{n\times n}$, $S_t^i \in \mathbb{R}^{m_t\times n}$, $R_t^i\in\mathbb{R}^{m_t\times m}$, and vectors $q_t^i\in\mathbb{R}^n$, $r_t^i\in\mathbb{R}^{m_t}$. For notational purposes, let the terms $R_t^i$, $S_t^i$, and $r_t^i$ be comprised of sub-matrices, $R_t^{i,j,k}\in\mathbb{R}^{m_t^j\times m_t^k}$, $S_t^{i,j}\in\mathbb{R}^{m_t^j\times n}$, and sub-vectors $r_t^{i,j}\in\mathbb{R}^{m_t^j}$, for $j,k\in\mathbf{N}$: 
\begin{equation}
\begin{aligned}
    &R_t^i := \begin{bmatrix} R_t^{i,1,1} & \hdots & R_t^{i,1,N}\\ \vdots & \ddots & \vdots \\ R_t^{i,N,1} & \hdots & R_t^{i,N,N} \end{bmatrix}, \ \ \ \  S_t^i := \begin{bmatrix} S_t^{i,1} \\ \vdots \\ S_t^{i,N}\end{bmatrix}, \ \ \ \ r_t^i := \begin{bmatrix} r_t^{i,1} \\ \vdots \\ r_t^{i,N} \end{bmatrix}
    \end{aligned}
\end{equation}
We additionally make use of the following matrix terms for brevity, which combine components from the cost functionals of all players:
\begin{equation}
\begin{alignedat}{5}
    &R_t := \begin{bmatrix} R_t^{1,1,1} & \hdots & R_t^{1,1,N} \\
    R_t^{2,2,1} & \hdots & R_t^{2,2,N} \\
    \vdots & \ddots & \vdots \\
    R_t^{N,N,1} & \hdots & R_t^{N,N,N}
    \end{bmatrix}, \ \ \ \ 
    &&S_{x_t} := \begin{bmatrix} S_t^{1,1} \\S_t^{2,2} \\ \vdots \\ S_t^{N,N} \end{bmatrix}, \ \ \ \ \ \ \ \ 
    &&r_t := \begin{bmatrix} r_t^{1,1} \\ r_t^{2,2} \\ \vdots \\ r_t^{N,N} \end{bmatrix}, \\
    &Q_t := \begin{bmatrix} Q_t^1 \\ \vdots \\ Q_t^N \end{bmatrix}, \ \ \ \ &&S_{u_t} := \begin{bmatrix} S_t^{1\intercal} \\ \vdots \\ S_t^{N\intercal}  \end{bmatrix}, \ \ \ \ \ \  \ \ &&q_t := \begin{bmatrix} q_t^1 \\ \vdots \\ q_t^N \end{bmatrix}, \\
    &\tilde{R}_t := \mathrlap{\begin{bmatrix} R_t^{1,2,1} & \hdots & R_t^{1,2,N} \\
     \vdots & \ddots & \vdots \\
     R_t^{1,N,1} & \hdots & R_t^{1,N,N} \\
     & & & \ddots \\ 
     & & & & R_t^{N,1,1} & \hdots & R_t^{N,1,N} \\
     & & & & \vdots & \ddots & \vdots \\
     & & & & R_t^{N,N-1,1} & \hdots & R_t^{N,N-1,N} \end{bmatrix}, 
     \tilde{S}_{x_t} := \begin{bmatrix} S_t^{1,2} \\ \vdots \\ S_t^{1,N} \\ S_t^{2,1} \\ S_t^{2,3} \\ \vdots \\ S_t^{2,N} \\ \vdots \\ S_t^{N,N-1} \end{bmatrix} }
\end{alignedat}
\end{equation}
We impose the regularity assumptions 
\begin{equation} 
R_t^{i,i,i} \succ 0, \ \  R_t^i \succeq 0, \ \ Q_t^i \succeq 0
\end{equation} 
to ensure that the objective of each player is strictly convex. These conditions are sufficient for any solution to the conditions \cref{eq:foncs} to constitute a GFNE, as stated in \cref{thm:sufficiency}, but not necessary.

The constraints imposed upon each player take the form
\begin{equation}
    \label{eq:linear_eq_constraints}
    \begin{aligned}
    0 &= H_{x_t}^i x_t + H_{u_t^1}^i u_t^1 + ... + H_{u_t^N}^i u_t^N + h_t^i, \ \ t\in\mathbf{T} \\
    0 &= H_{x_{T+1}}^i x_{T+1} + h_{T+1}^i,
    \end{aligned}
\end{equation}
for matrices $H_{x_t}^i \in \mathbb{R}^{a_t^i\times n}$, $H_{u_t^j}^i \in \mathbb{R}^{a_t^i\times m_t^j}$, and vectors $h_t^i\in\mathbb{R}^{a_t^i}$, where $a_t^i$ is the dimension of the equality constraint imposed on player $i$ at stage $t\in\mathbf{T^+}$. As in \cref{sec:formulation}, we associate multipliers $\mu_{t}^i \in\mathbb{R}^{a_t^i}$, $t\in\mathbf{T^+}$ with these constraints for each player $i\in\mathbf{N}$.  
Let 
\begin{equation}
\begin{alignedat}{3}
H_{u_t} &:= \begin{bmatrix} H_{u_t^1}^1 & \hdots & H_{u_t^N}^1 \\ \vdots & \ddots & \vdots \\ H_{u_t^1}^N & \hdots & H_{u_t^N}^N \end{bmatrix}, \ \ &&\hat{H}_{u_t} := \begin{bmatrix} H_{u_t^1}^1 \\ & \ddots \\ & & H_{u_t^N}^N \end{bmatrix}, \\
H_{x_t} &:= \begin{bmatrix} H_{x_t}^1 \\ \vdots \\ H_{x_t}^N \end{bmatrix},  \ \ \ &&\hat{H}_{x_t}:= \begin{bmatrix} H_{x_t}^1 \\ & \ddots \\ & & H_{x_t}^N \end{bmatrix}, \\
h_t^\intercal &:=\begin{bmatrix} (h_t^{1})^\intercal & \hdots & (h_t^N)^\intercal \end{bmatrix}, &&\mu_t^\intercal := \begin{bmatrix} (\mu_t^1)^\intercal & \hdots & (\mu_t^N)^\intercal \end{bmatrix}, \\
\tilde{H}_{u_t} &:= \begin{bmatrix} H_{u_t^2}^1 & \hdots & H_{u_t^N}^1 \\ & & & \ddots \\
& & & & H_{u_t^1}^N & \hdots & H_{u_t^{N-1}}^N
\end{bmatrix}
\end{alignedat}
\end{equation}

Due to the linearity of all dynamic and non-dynamic constraint functions appearing in the game, and the quadratic cost functionals, the solutions of the conditions \cref{eq:foncs} and \cref{eq:qfoncs} will be identical, as stated in \cref{sec:formulation}. Therefore we will use the terms $K_t$ and $\nabla_{x_t}\pi_t(x_t)$ interchangeably in this section.

Using the above-defined dynamic, constraint and cost terms, we are able to proceed with development of numerical methods for computing GFNE solutions to this game. Instead of taking a dynamic programming perspective as is, for example, taken in the classic derivation of Feedback Nash Equilibria for unconstrained LQ games in \cite[Chapter 6]{basar1998dynamic}, we derive our method using what we refer to as a \emph{dynamic matrix factorization}. The primary idea behind this derivation is simply that the computation used to evaluate $K_{t+1}$ can be reused to compute $K_t$ efficiently. 

To start, the conditions \cref{eq:qfoncs} for stage $t=T$ can be expressed in terms of the following matrix system:
\begin{equation}
    \label{eq:lq_T}
    \begin{bmatrix}
    R_T & -\hat{H}_{u_T}^\intercal & \hat{B}_T^\intercal &  \\
    H_{u_T} \\
    -B_T & & & I_n \\
    & & -I_{N*n} & Q_{T+1} & \hat{H}_{x_{T+1}}^\intercal \\
    & & &H_{x_{T+1}}
    \end{bmatrix} \begin{bmatrix} u_T \\ \mu_T \\ \lambda_{T} \\ x_{T+1} \\ \mu_{T+1} \end{bmatrix} + \begin{bmatrix} S_{x_t} \\ H_{x_T} \\ -A_T \\ 0 \\ 0 \end{bmatrix} x_T + \begin{bmatrix} r_t \\ h_T \\ -c_T \\ q_{T+1} \\ h_{T+1} \end{bmatrix} = 0
\end{equation}

In \cref{eq:lq_T}, the matrices $I_\Box$ denote the $\Box\times \Box$-dimensional identity matrix.  Letting the system in \cref{eq:lq_T} be denoted in shorthand as 
\begin{equation} 
     M_T z_T + N_T x_T + n_T = 0,
\end{equation} where $z_T := [u_T^\intercal \ \mu_T^\intercal \ \lambda_T^\intercal \ x_{T+1}^\intercal \ \mu_{T+1}^\intercal]^\intercal$, we have that 
\begin{equation} 
\begin{alignedat}{3}
\pi(x_T) &:= K_T x_T + k_T, \\
\lambda_T &:= K_{\lambda_T} x_T + k_{\lambda_T}, \\
\mu_T &:= K_{\mu_T} x_T + k_{\mu_T}, \\
\mu_{T+1} &:= K_{\mu_{T+1}} x_T + k_{\mu_{T+1}}, \\
K_T &:= -[M_T^{-1}N_T]_{u_T}, \ \  && k_T := -[M_T^{-1}n_T]_{u_T}, \\
K_{\lambda_T} &:= -[M_T^{-1}N_T]_{\lambda_T}, \ \ && k_{\lambda_T} := -[M_T^{-1}n_T]_{\lambda_T}, \\
K_{\mu_T} &:= -[M_T^{-1}N_T]_{\mu_T}, \ \ && k_{\mu_T} := -[M_T^{-1}n_T]_{\mu_T}, \\
K_{\mu_{T+1}} &:= -[M_{T}^{-1}N_T]_{\mu_{T+1}}, \ \ && k_{\mu_{T+1}} := -[M_T^{-1}n_T]_{\mu_{T+1}},
\end{alignedat}
\end{equation}
For any stage $t$, we also make use of the matrix
\begin{equation}
    \begin{aligned}
    \Pi_t^\intercal := \begin{bmatrix} K_t^{2\intercal} & \cdots & K_t^{N\intercal} \\ & & & K_t^{1\intercal} & K_t^{3\intercal} & \cdots & K_t^{N\intercal} \\ & & & & & & & \ddots \\ & & & & & & & & K_t^{1\intercal} & \cdots & K_t^{(N-1)\intercal}
    \end{bmatrix}^\intercal
    \end{aligned}
\end{equation}


Notice that if we denote the conditions \cref{eq:qfoncs} for some stage $t+1$ as 
\begin{equation}
    M_{t+1} z_{t+1} + N_{t+1} x_{t+1} + n_{t+1} = 0,
\end{equation}
as we did for $t+1=T$, then we have also that the conditions \cref{eq:qfoncs} for stage $t$ can be denoted as 
\begin{equation}
    M_{t} z_{t} + N_{t} x_{t} + n_{t} = 0,
\end{equation}
where $z_t := [u_t^\intercal \ \mu_t^\intercal \ \lambda_t^\intercal \ \psi_{t+1}^\intercal \ x_{t+1}^\intercal \ z_{t+1}^\intercal]^\intercal$, and $M_t$, $N_t$, and $n_t$ are defined as:
\begin{equation}
    \label{eq:matrix_sys_t}
    \begin{aligned}
    M_t &:= \begin{bmatrix} \multicolumn{1}{c}{D_t^{1}} & D_t^{2} \\ \begin{bmatrix} 0 & N_{t+1} \end{bmatrix} & M_{t+1} \end{bmatrix} \\
    D_t^{1} &:= \begin{bmatrix}
    R_t & -\hat{H}_{u_t}^\intercal & \hat{B}_t^\intercal & \\
    H_{u_t} \\
    -B_{t} & & & & I_n \\
    & & -I_{N*n} & \Pi_{t+1}^\intercal & Q_{t+1} \\ 
    & & & -I_{(N-1)*m_{t+1}} & \tilde{S}_{x_{t+1}} 
    \end{bmatrix}, \\
    D_t^{2} &:=\begin{bmatrix}
    & \\
    & \\
    S_{u_{t+1}} & -\hat{H}_{x_{t+1}}^\intercal & \hat{A}_{t+1}^\intercal & & & & & & & & & \\ 
    \tilde{R}_{t+1} & -\tilde{H}_{u_{t+1}}^\intercal & \tilde{B}_{t+1}^\intercal & & & & & &  \\
    \end{bmatrix}, \\
    N_t^\intercal &:= \begin{bmatrix} S_{x_{t}}^\intercal &
    H_{x_{t}}^\intercal &
    -A_{t}^\intercal &  & \ \end{bmatrix}, \ \  n_t^\intercal := \begin{bmatrix} r_t^\intercal & h_t^\intercal & -c_t^\intercal  & 0 & 0 & n_{t+1}^\intercal \end{bmatrix}. 
    \end{aligned}
\end{equation}
From this form, we have as before, that 
\begin{equation} 
\label{eq:policies_lq_t}
\begin{alignedat}{3}
\pi(x_t) &:= K_t x_t + k_t, \\
\lambda_t &:= K_{\lambda_t} x_t + k_{\lambda_t}, \\
\mu_t &:= K_{\mu_t} x_t + k_{\mu_t}, \\
\psi_{t+1} &:= K_{\psi_{t+1}} x_t + k_{\psi_{t+1}}, \\
K_t &:= -[M_t^{-1}N_t]_{u_t}, \ \  && k_t := -[M_t^{-1}n_t]_{u_t}, \\
K_{\lambda_t} &:= -[M_t^{-1}N_t]_{\lambda_t}, \ \ && k_{\lambda_t} := -[M_t^{-1}n_t]_{\lambda_t}, \\
K_{\psi_{t+1}} &:= -[M_t^{-1}N_t]_{\psi_{t+1}}, \ \ && k_{\psi_{t+1}} := -[M_t^{-1}n_t]_{\psi_{t+1}}, \\
K_{\mu_t} &:= -[M_t^{-1}N_t]_{\mu_t}, \ \ && k_{\mu_t} := -[M_t^{-1}n_t]_{\mu_t},
\end{alignedat}
\end{equation}
The advantage of expressing our system in the form \cref{eq:matrix_sys_t} is that the computation performed to solve $K_{t+1}$ and $k_{t+1}$ can be reused to solve $K_t$ and $k_t$. Using the block form of $M_t$ defined in \cref{eq:matrix_sys_t}, and pre suming $M_{t+1}$ is non-singular, we have from \cite{lu2002inverses}:

\begin{equation}
\begin{aligned}
    [M_t^{-1}[N_t \ n_t]]_{u_t,\mu_t,\lambda_t,\psi_t,x_{t+1}} &=
    \begin{bmatrix} P_t^1 & P_t^2 \end{bmatrix} \begin{bmatrix} N_t & n_t \end{bmatrix}, \\
    &= \begin{bmatrix} P_t^1 N_t, & P_t^1 \begin{bmatrix} r_t \\ h_t \\ -c_t \\  \\ \ \end{bmatrix} + P_t^2 n_{t+1} \end{bmatrix}
\end{aligned}
\end{equation}
where the matrices $P_t^1$ and $P_t^2$ are defined as 
\begin{equation}
    \begin{aligned}
    P_t^1 &:= \big(D_t^1 - D_t^2M_{t+1}^{-1}[0 \ N_{t+1}]\big)^{-1} \\
    P_t^2 &:= -P_t^1 D_t^2 M_{t+1}^{-1}
    \end{aligned}
\end{equation}
Substituting in the form of the matrices in \cref{eq:matrix_sys_t}, we have that 
\begin{equation}
    \begin{aligned}
    P_t^1 &:= \begin{bmatrix}
    R_t & -\hat{H}_{u_t}^\intercal & \hat{B}_t^\intercal & \\
    H_{u_t} \\
    -B_{t} & & & & I_n \\
    & & -I_{N*n} & \Pi_{t+1}^\intercal & \hat{P}_t^{1,a} \\ 
    & & & -I_{(N-1)*m_{t+1}} & \hat{P}_t^{1,b}  
    \end{bmatrix}^{-1}, \\
    \hat{P}_t^{1,a} &:= Q_{t+1}-S_{u_{t+1}}K_{t+1}+\hat{H}_{x_{t+1}}^\intercal K_{\mu_{t+1}} - \hat{A}_{t+1}^\intercal K_{\lambda_{t+1}}, \\
    \hat{P}_t^{1,b} &:= \tilde{S}_{x_{t+1}} - \tilde{R}_{t+1}K_{t+1} + \tilde{H}_{u_{t+1}}^\intercal K_{\mu_{t+1}} - \tilde{B}_{t+1}^\intercal K_{\lambda_{t+1}}, \\
    P_t^2 n_{t+1} &:= -P_t^1 \begin{bmatrix}  \\   S_{u_{t+1}}k_{t+1} - \hat{H}_{x_{t+1}}^\intercal k_{\mu_{t+1}} + \hat{A}_{t+1}^\intercal k_{\mu_{t+1}} \\ 
    \tilde{R}_{t+1} k_{t+1} - \tilde{H}_{u_{t+1}}^\intercal k_{\mu_{t+1}} + \tilde{B}_{t+1}^\intercal k_{\lambda_{t+1}}
    \end{bmatrix}.
    \end{aligned}
\end{equation}
From the above, it can be seen that the entire matrix inverse $M_t^{-1}$ does not need to be computed for any stage $t$ (other than the terminal stage $T$), and the factorization presented here allows computation of the entire GFNE trajectory and associated multipliers very efficiently. In particular, the overall computational complexity of solving this system is $O(T\cdot((N+1)\cdot(n+m))^2\cdot(n+1))$ time due to the dominating cost of at each stage solving the system of equations of the form $P_t^1 W_t$, where $(P_t^1)^{-1}$ is a square matrix of width no greater than $(N+1)\cdot(n+m)$, and $W_t$ is some matrix with $n+1$ columns. 

After computing the terms \cref{eq:policies_lq_t} for all stages $t$ using the procedure above, the resultant GFNE trajectory and associated multipliers can be extracted:

\begin{equation}
\label{eq:extract_solution_lq}
\begin{alignedat}{3}
    x_1^* &:= \hat{x}_1, \\
    u_s^* &:= K_s x_s^* + k_s, &&s\in\mathbf{T}, \\
    \mu_s &:= K_{\mu_s}x_s^* + k_{\mu_s}, &&s\in\mathbf{T}, \\
    \lambda_s &:= K_{\lambda_s} x_s^* + k_{\lambda_s}, &&s\in\mathbf{T}, \\
    \psi_{s} &:= K_{\psi_{s}} x_{s-1}^* + k_{\psi_{s}}, &&s\in\mathbf{T_2}, \\
    x_{s+1}^* &:= A_s x_s^* + B_s u_s^* + c_s, \ \  &&s\in\mathbf{T}, \\
    \mu_{T+1} &:= K_{\mu_{T+1}}x_T + k_{\mu_{T+1}}.
\end{alignedat}
\end{equation}

\section{Inequality-Constrained LQ Games} 
\label{sec:inequality_lq}
We now extend the basic results presented in \cref{sec:equality_lq} on the computation of GFNE for equality-constrained LQ games, to the computation of GFNE for inequality-constrained LQ games. The approach we take here is that of an active-set method, analogous to active-set methods for quadratic programming (see, e.g. \cite{nocedal2006numerical}).

Consider a dynamic game among $N$ players over $T$ stages, with linear dynamics described by \cref{eq:linear_dynamics}, and quadratic cost functionals \cref{eq:quadratic_cost_functional}. Assume that each player is also subject to linear equality constraints of the form \cref{eq:linear_eq_constraints}, along with linear inequality constraints of the form 
\begin{equation}
    \label{eq:linear_ineq_constraints}
    \begin{aligned}
    0 &\leq G_{x_t}^i x_t + G_{u_t^1}^i u_t^1 + ... + G_{u_t^N}^i u_t^N + g_t^i, \ \ t\in\mathbf{T} \\
    0 &\leq G_{x_{T+1}}^i x_{T+1} + g_{T+1}^i,
    \end{aligned}
\end{equation}
for matrices $G_{x_t}^i \in \mathbb{R}^{b_t^i\times n}$, $G_{u_t^j}^i \in \mathbb{R}^{b_t^i\times m_t^j}$, and vectors $g_t^i\in\mathbb{R}^{b_t^i}$, where $b_t^i$ is the dimension of the inequality constraint imposed on player $i$ at stage $t\in\mathbf{T^+}$. As in \cref{sec:formulation}, we associate multipliers $\gamma_{t}^i \in\mathbb{R}^{b_t^i}$, $t\in\mathbf{T^+}$ with these constraints for each player $i\in\mathbf{N}$. Assume that a solution to the system \cref{eq:foncs} exists for this game at stage $t=1$, and that strict complementarity holds for the conditions \cref{eq:foncs} for the subgame starting at every $t\in\mathbf{T}$ along any solution. 

The method we present for computing a GFNE of this game is an adaptation of Algorithm 16.3 in \cite{nocedal2006numerical} to the current setting. Under the strict complementarity assumption (which ensures differentiability of the policies $\pi_t$ along the solution), we have that at any GFNE solution, some subset of the constraints \cref{eq:linear_ineq_constraints} associated with strictly positive multipliers hold with equality at the solution. If the set of active constraints along some solution were known in advance, we could consider all active constraints as equality constraints, ignore all inactive constraints, and solve for the resultant equality-constrained game using the method presented in \cref{sec:equality_lq}. In general the set of active constraints along a solution is obviously unknown in advance. The active-set method we propose accounts for this by iteratively solving for the unique GFNE solution for different guesses of the active constraint set, and uses dual variable information to update the guess of the active set. In the remainder of this section we describe the proposed method. The presentation of this section is based off of section 16.5 in \cite{nocedal2006numerical}, with necessary modifications made to account for the multiplayer feedback setting considered here.  

The method begins with a feasible initialization for the game (defined by the linear dynamics \cref{eq:linear_dynamics}, equality constraints \cref{eq:linear_eq_constraints}, and inequality constraints \cref{eq:linear_ineq_constraints}, and the quadratic cost functionals \cref{eq:quadratic_cost_functional}). We denote the set of all primal variables associated with the game at the $k$th iteration of the method by $\mathbf{X}_k := [x_{(1:T+1),k}, u_{(1:T),k}]$. Also associated with the $k$th iteration of the algorithm is the working set $\mathcal{W}_k$ which denotes the set of constraints which are treated with equality at the $k$th iteration. The working set $\mathcal{W}_k$ always contains all of the equality constraints \cref{eq:linear_eq_constraints}. The working set $\mathcal{W}_1$ is taken to be a subset of the constraints active along the initialization $\mathbf{X}_1$.

Given an iterate $\mathbf{X}_k$ and working set $\mathcal{W}_k$, we find a step $\mathbf{P}_k := [p_{x_{1:T+1}}, p_{u_{1:T}}]$ is to move $\mathbf{X}_k$ to the GFNE associated with the working set of equality constraints in $\mathcal{W}_k$. Specifically, the problem to be solved at each iteration is the GFNE problem for the equality-constrained LQ game defined by each player's stage-wise cost functionals 
\begin{equation}
    \label{eq:quadratic_cost_functional_active}
    \begin{alignedat}{3}
    l^i_{t,k}(p_{x_t},p_{u_t}) &:= \frac{1}{2}\bigg( \begin{bmatrix} p_{x_t} \\p_{u_t} \end{bmatrix}^\intercal \begin{bmatrix} Q_t^i & S_t^{i\intercal} \\ S_t^i & R_t^i\end{bmatrix} \begin{bmatrix} p_{x_t} \\p_{u_t}\end{bmatrix} + 2\begin{bmatrix} p_{x_t}\\p_{u_t}\end{bmatrix}^\intercal \begin{bmatrix} q_{t,k}^i \\ r_{t,k}^i \end{bmatrix} \bigg), \ \ &&t\in\mathbf{T}  \\
        q_{t,k}^i &:= Q_t^i x_{t,k} + S_t^{i\intercal}u_{t,k} + q_t^i, \ &&t\in\mathbf{T}, \\
    r_{t,k}^i &:= S_t^i x_{t,k} + R_t^i u_{t,k} + r_t^i, && t\in\mathbf{T}, \\
    l^i_{T+1,k}(p_{x_{T+1}}) &:=\frac{1}{2} \bigg(p_{x_{T+1}}^\intercal Q_{T+1}^i p_{x_{T+1}} + 2p_{x_{T+1}}^\intercal q_{T+1,k}^i\bigg), \\
    q_{T+1,k}^i &:= Q_{T+1}^i x_{T+1,k} + q_{T+1}^i,
    \end{alignedat}
\end{equation}
the dynamics
\begin{equation}
    \label{eq:linear_dynamics_active}
    p_{x_{t+1}} = A_t p_{x_t} + B_t^1 p_{u_t^1} + ... + B_t^N p_{u_t^N} , \ \ t\in\mathbf{T},
\end{equation}
and the linear equality constraints 
\begin{equation}
    \label{eq:linear_eq_constraints_active}
    \begin{aligned}
    0 &= H_{(x_t,k)}^i p_{x_t} + H_{(u_t^1,k)}^i p_{u_t^1} + ... + H_{(u_t^N,k)}^i p_{u_t^N}, \ \ t\in\mathbf{T} \\
    0 &= H_{(x_{T+1},k)}^i p_{x_{T+1}}.
    \end{aligned}
\end{equation}

Above, the matrices $H^i_{(x_t,k)}$ and $H^i_{(u_t^j,k)}$ are defined to be the set of active equality constraint coefficients corresponding to $\mathcal{W}_k$:

\begin{equation}
    \label{eq:active_eqs}
    H^i_{(x_t,k)} := \mleft[ \begin{array}{c} H^i_{x_t} \\ \hline  \vdots \\ \{ G_{x_t}^{i,j} \}_{(t,i,j)\in\mathcal{W}_k \cap \mathcal{I}}  \\ \vdots\end{array} \mright], H_{(u_t,k)}^i := \mleft[ \begin{array}{c} H^i_{u_t} \\ \hline  \vdots \\ \{ G_{u_t}^{i,j} \}_{(t,i,j)\in\mathcal{W}_k \cap \mathcal{I}}  \\ \vdots\end{array}\mright].
\end{equation}
Here, $\mathcal{W}_k\cap \mathcal{I}$ is the index set of all active inequality constraints, and $G_{x_t}^{i,j}$ is the $j$th row of the matrix $G_{x_t}^i$.

Associated with the constraints \cref{eq:linear_eq_constraints_active} are multipliers $\mu_{(t,k)}^i$, defined as 
\begin{equation}
    \mu_{(t,k)}^i := \mleft[ \begin{array}{c}
    \mu_t^i \\ \hline \vdots \\ \{\gamma_t^{i,j}\}_{(t,i,j)\in\mathcal{W}_k\cap\mathcal{I}} \\ \vdots
    \end{array} \mright]
\end{equation}
where $\gamma_t^{i,j}$ is the $j$th element of $\gamma_t^i$. 

After solving for $\mathbf{P}_k$, the GFNE of the resultant equality-constrained LQ game, $\mathbf{X}_k + \mathbf{P}_k$ is the GFNE for the equality-constrained LQ game defined by \cref{eq:linear_dynamics}, \cref{eq:quadratic_cost_functional}, \cref{eq:linear_eq_constraints}, and the active constraints \cref{eq:linear_ineq_constraints} in $\mathcal{W}_k\cap\mathcal{I}$. However, it may be that $\mathbf{X}_k + \mathbf{P}_k$ is infeasible with respect to the entire set of inequality constraints \cref{eq:linear_ineq_constraints}. Therefore we instead find the point $\mathbf{X}_k + \beta_k \mathbf{P}_k$, where 
\begin{equation}
    \label{eq:find_alpha}
    \begin{aligned}
    \beta_k := &\max_{\beta\in[0,1]} \beta \\
    \text{s.t.   } \mathbf{X}_k + &\beta \mathbf{P}_k \text{ feasible w.r.t. \cref{eq:linear_ineq_constraints}}.
    \end{aligned}
\end{equation}
 The optimization in \cref{eq:find_alpha} is a linear program, and $\beta_k$ can be computed exactly and efficiently. When $\beta_k < 1$, it implies that there is an inequality constraint not considered in the working set which must be accounted for. When this is the case, the iterate $\mathcal{X}_{k+1}$ is updated to the point $\mathbf{X}_k + \beta_k \mathbf{P}_k$, and the working set is updated to include the blocking constraint. 
If instead, $\beta_k = 1$, then the point $\mathbf{X}_k + \mathbf{P}_k$ both is a GFNE solution for the working set and is feasible with respect to all equality and inequality constraints. All that is left to check is whether the constraints $\mu_{(t,k)}^{i,j} > 0$ for all $j > a_t^i$, meaning the complementarity conditions in \cref{eq:foncs} are satisfied, and therefore a solution satisfying the entire set of conditions \cref{eq:foncs} for the inequality-constrained problem has been found. 
If some multiplier associated with the inequality constraints in the working set $\mathcal{W}_k$ is negative\footnote{The multipliers cannot be zero by the strict complementarity assumption made.}, then the corresponding constraint is to be dropped from the working set at the next iteration. Unlike the the convex quadratic programming setting, where forming $\mathcal{W}_{k+1}$ by dropping a constraint associated with a negative multiplier in the set $\mathcal{W}_k\cap\mathcal{I}$ and setting $\mathbf{X}_{k+1}=\mathbf{X}_k +\mathbf{P}_k$, the update $\mathbf{P}_{k+1}$ does not necessarily move away from the dropped constraint boundary. In such situations, the procedure fails to make progress, since the value of $\beta_{k+1}$ in \cref{eq:find_alpha} will be $0$. In practice these conditions can be treated by dropping a different constraint (also associated with a negative multiplier) from the working set. If no other constraints are associated with negative multipliers, a GFNE does not exist in the vicinity of the iterate, and failure is declared.  The full procedure is stated in \cref{alg:lq_ineq}.

\begin{algorithm}
\caption{Active Set Inequality Constrained LQ Game GFNE Solver}
\label{alg:lq_ineq}
\begin{algorithmic}[1]
\STATE{Start with $\mathbf{X}_1$ feasible w.r.t. \cref{eq:linear_dynamics}, \cref{eq:linear_eq_constraints}, \cref{eq:linear_ineq_constraints}}
\STATE{Initialize $\mathcal{W}_1$ to be subset of active inequality constraints along $\mathbf{X}_1$}
\FOR{k=1,2,3,...}
\STATE{Solve equality-constrained LQ GFNE defined by \cref{eq:quadratic_cost_functional_active}, \cref{eq:linear_dynamics_active}, \cref{eq:linear_eq_constraints_active}, and denote the solution as $\mathbf{P}_k$}
\IF{$\mathbf{P}_k == \mathbf{0}$}
    \STATE{Extract multipliers $\mu_{(t,k)}^i$, $\psi_t^i$, $\lambda_t^i$, using \cref{eq:extract_solution_lq}}
    \IF{$\mu_{(t,k)}^{i,j} > 0, \ \forall j > a_t^i$ \COMMENT{Inequality constraint multipliers}}
        \RETURN $\mathbf{X}_k$, 
    \ELSE
    \STATE\label{line10}{$(t_m,i_m,j_m)\gets \underset{(t,i,j)\in\mathcal{W}_k\cap\mathcal{I}}{\text{argmin}} \mu_{(t,k)}^{i,j}$}
    \STATE{$\mathbf{X}_{k+1}\gets\mathbf{X}_k$}
    \STATE\label{line12}{$\mathcal{W}_{k+1}\gets\mathcal{W}_k\setminus\{(t_m,i_m,j_m)\}$}
    \ENDIF
\ELSE
    \STATE{Find largest $\beta_k\in[0,1]$ such that $\mathbf{X}_k + \beta_k \mathbf{P}_k$ is feasible w.r.t. \cref{eq:linear_ineq_constraints}}
    \STATE{$\mathbf{X}_{k+1} \gets \mathbf{X}_k + \beta_k\mathbf{P}_k$}
    \IF{$\beta_k < 1$}
    \STATE{$(t_b,i_b,j_b) \gets $ index of blocking inequality constraint not already in $\mathcal{W}_k\cap\mathcal{I}$}
    \IF{$(t_b,i_b,j_b) == (t_m,i_m,j_m)$ \COMMENT{Blocking index is previously dropped index}}
    \STATE\label{cycle}{Choose other blocking constraint or declare failure}
    \ELSE
    \STATE{$\mathcal{W}_{k+1}\gets\mathcal{W}_k\cup (t_b,i_b,j_b)$}
    \ENDIF
    \ELSE
    \STATE{$\mathcal{W}_{k+1}\gets\mathcal{W}_k$}
    \ENDIF
\ENDIF
\ENDFOR
\end{algorithmic}
\end{algorithm}

\section{Nonlinear Games}
\label{sec:nonlinear_games}

We now present a method for computing GFQNE solutions to general dynamic games. In particular, we outline a procedure for finding a solution to the conditions \cref{eq:qfoncs} for games defined by the dynamics \cref{eq:dynamics}, cost functionals \cref{eq:cost_functional}, and constraints \cref{eq:equality_constraints,eq:inequality_constraints}. We assume that all functions appearing in the conditions \cref{eq:qfoncs} are continuously twice differentiable, with the exception of the implicitly defined policies. The procedure leverages the method for computing solutions to inequality-constrained LQ games presented in \cref{sec:inequality_lq} and \cref{alg:lq_ineq}, and generally is inspired by Sequential Quadratic Programming methods for non-convex numerical optimization (see e.g. \cite{nocedal2006numerical}, Chapter 18). 

The foundation of this approach is in the observation that a Newton-style method can be used to finding a solutions to the conditions \cref{eq:qfoncs}, where each iteration involves solving for a GFNE for the locally approximate LQ game formed around the current iterate. Considering first the case in which the game does not include any inequality constraints, computing a search direction using Newton's method on the conditions \cref{eq:qfoncs} at $t=1$ can be seen to be equivalent to computing a search direction by solving an LQ approximation of the game. In the inequality-constrained case, search directions can be found by solving an inequality-constrained LQ approximation, analogous to the method in \cite{nocedal2006numerical}. 

In particular, we propose an iterative method for finding solutions to the conditions \cref{eq:qfoncs}. Throughout this section, iterations are again indexed by $k=1,2,3,...$, as they were in the method for computing inequality-constrained LQ games in \cref{sec:inequality_lq}. Here, let the current iterate of the primal and dual game variables at iteration $k$ be denoted as
\begin{equation} 
\begin{aligned}
\mathbf{X}_k &:= [x_{(1:T+1),k},u_{(1:T),k}], \\
\mathbf{\Lambda}_k &:= [\lambda_{(1:T),k}, \mu_{(1:T+1),k}, \gamma_{(1:T+1),k}, \psi_{(2:T),k}] 
\end{aligned}
\end{equation}
 Assume some initialization of all variables $\mathbf{X}_1$ and $\mathbf{\Lambda}_1$. Then for each iteration $k$, a search direction $\mathbf{P}_k$ is found by solving the inequality-constrained LQ game formed as follows. Let the dynamics \cref{eq:linear_dynamics} for the approximate game at the $k$th iteration be defined by the terms 
\begin{equation}
\label{eq:linear_dyn_iter_k}
\begin{aligned}
    A_{t,k} &:= \nabla_{x} f_t(x_{t,k},u_{t,k}), \\
    B^i_{t,k} &:= \nabla_{u^i} f_t(x_{t,k},u_{t,k}), \\
    c_{t,k} &:= f_t(x_{t,k},u_{t,k}) - x_{t+1,k}.
\end{aligned}
\end{equation}
Similarly, let the equality and inequality constraint terms in  \cref{eq:linear_eq_constraints,eq:linear_ineq_constraints} be defined as 
\begin{equation}
\label{eq:linear_constraint_iter_k}
\begin{alignedat}{3}
H_{x_t,k}^i &:= \nabla_x   h_t^i(x_{t,k},u_{t,k}), \ \  &&G_{x_t,k}^i := \nabla_x   g_t^i(x_{t,k},u_{t,k}),\\
H_{u_t^i,k}^i &:= \nabla_{u^i} h_t^i(x_{t,k},u_{t,k}), \ \  &&G_{u_t^i,k}^i := \nabla_{u^i}   g_t^i(x_{t,k},u_{t,k}), \\
h_{t,k}^i &:= h_t^i(x_{t,k},u_{t,k}), \ \  && \ \ g_{t,k}^i :=    g_t^i(x_{t,k},u_{t,k}).
\end{alignedat}
\end{equation}

Finally, let the cost functional coefficients in \cref{eq:quadratic_cost_functional} for stages $t\in\mathbf{T}$ be defined as 
\begin{equation}
    \label{eq:quadratic_cost_iter_k}
    \begin{alignedat}{3}
    Q_{t,k}^i &:=  \nabla^2_{x,x}l_t^i(x_{t,k},u_{t,k}) + (\nabla^2_{x,x} f_t)^\intercal \lambda_{t,k}^i - (\nabla^2_{x,x} h_t^i)^\intercal \mu_{t,k}^i - (\nabla^2_{x,x} g_t^i)^\intercal \gamma_{t,k}^i, \\  
    S_{t,k}^i &:= \nabla^2_{u,x}l_t^i(x_{t,k},u_{t,k})+ (\nabla^2_{u,x} f_t)^\intercal \lambda_{t,k}^i - (\nabla^2_{u,x} h_t^i)^\intercal \mu_{t,k}^i - (\nabla^2_{u,x} g_t^i)^\intercal \gamma_{t,k}^i, \\ 
    R_{t,k}^i &:= \nabla^2_{u,u}l_t^i(x_{t,k},u_{t,k})+ (\nabla^2_{u,u} f_t)^\intercal \lambda_{t,k}^i - (\nabla^2_{u,u} h_t^i)^\intercal \mu_{t,k}^i - (\nabla^2_{u,u} g_t^i)^\intercal \gamma_{t,k}^i, \\
    q_{t,k}^i &:= \nabla_x l_t^i(x_{t,k},u_{t,k}), \\
    r_{t,k}^i &:= \nabla_u l_t^i(x_{t,k},u_{t,k}).
    \end{alignedat}
\end{equation}

The solution to this inequality-constrained LQ game at each iteration $k$ yields the search direction $\mathbf{P}_k$ and multipliers $\bar{\mathbf{\Lambda}}_{k+1}$. To ensure progress towards a solution of the conditions \cref{eq:qfoncs}, a line-search procedure is invoked. We seek a parameter $\alpha_k \in [0,1]$ such that the iterates \begin{equation}
\begin{aligned}
\mathbf{X}_{k+1} &:= \mathbf{X}_k + \alpha_k \mathbf{P}_k, \\ 
\mathbf{\Lambda}_{k+1} &:= \mathbf{\Lambda}_k + \alpha_k (\bar{\mathbf{\Lambda}}_{k+1} - \mathbf{\Lambda}_k) 
\end{aligned}
\end{equation} make maximal progress towards satisfying the conditions \cref{eq:qfoncs}, with respect to an appropriate merit function. Because the game setting involves multiple players, a decrease in the objective of all players' objectives cannot be guaranteed. Therefore the merit function we consider is simply the residual squared norm of the conditions \cref{eq:qfoncs}. In particular, define the merit function to search over as the following:
\begin{equation}
    \label{eq:merit_func}
    \begin{alignedat}{3}
    \mathbf{M}(\mathbf{X}, \mathbf{\Lambda}) := & \sum_{i\in\mathbf{N}} \sum_{s\in\mathbf{T}} \ 
    && \left\| \nabla_{u_s^i} \bigg[ l_s^{i} + f_s^\intercal \lambda_{s}^i - h_s^{i\intercal} \mu_{s}^i  -  g_s^{i^\intercal} \gamma_{s}^i    \bigg] \right\|_2^2 + \\
    &\sum_{i\in\mathbf{N}} \sum_{s\in\mathbf{T_2}} \ && \left\| \nabla_{x_s} \bigg[ l_s^i - \lambda_{s-1}^i + f_s^\intercal \lambda_{s}^i - h_s^{i^\intercal} \mu_s^i  -  g_s^{i\intercal} \gamma_{s}^i \bigg]  + \tilde{\nabla}_{x_s} \hat{\pi}_s^{-i\intercal} \psi_{s}^i \right\|_2^2 + \\
     &\sum_{i\in\mathbf{N}} \sum_{s\in\mathbf{T_2}} \ && \left\| \nabla_{u^{-i}_s} \bigg[ l_s^i + f_s^\intercal \lambda_{s}^i - h_s^{i^\intercal} \mu_s^i  -  g_s^{i\intercal} \gamma_s^i  - \psi_s^i \bigg] \right\|_2^2 + \\
     &\sum_{i\in\mathbf{N}} \ && \left\| \nabla_{x_{T+1}} \bigg[ l_{T+1}^i - \lambda_{T}^i - h_{T+1}^{i^\intercal} \mu_{T+1}^i  -  g_{T+1}^{i\intercal} \gamma_{T+1}^i \bigg] \right\|_2^2 + \\
     & \sum_{s\in\mathbf{T}} \ && \left\| x_{s+1} - f_s(x_{s},u_{s}) \right\|_2^2 + \\
      & \sum_{i\in\mathbf{N}} \Big( \sum_{s\in\mathbf{T}} \ && \left\| h_s^i(x_{s},u_{s}) \right\|_2^2 + \left\| h_{T+1}^i(x_{T+1}) \right\|_2^2 \Big) + \\
      & \sum_{i\in\mathbf{N}} \Big( \sum_{s\in\mathbf{T}} \ && \left\| \min(g_s^i(x_{s},u_{s}),0) \right\|_2^2 + \left\| \min(g_{T+1}^i(x_{T+1}),0) \right\|_2^2 \Big) + \\
      & \sum_{i\in\mathbf{N}} \sum_{s\in\mathbf{T^+}} \ && \left\| \min(\mu_{s}^i,0) \right\|_2^2 + |(g_s^i)^\intercal \mu_{s}^i|.
    \end{alignedat}
\end{equation}

 Recall that $\mathbf{T}_2 := \{2,...,T\}$. Then $\alpha_k$ is defined as:
\begin{equation}
    \label{eq:line_search}
    \begin{aligned}
    \alpha_k := \min_{\alpha\in[0,1]} \mathbf{M}(\mathbf{X}_k + \alpha \mathbf{P}_k, \mathbf{\Lambda}_k + \alpha (\bar{\mathbf{\Lambda}}_{k+1} - \mathbf{\Lambda}_k)).
    \end{aligned}
\end{equation}

\begin{algorithm}
\caption{GFQNE Solver for Nonlinear Games}
\label{alg:nlgames}
\begin{algorithmic}[1]
\STATE{Set convergence tolerance $\epsilon > 0$}
\STATE{Start with initial $\mathbf{X}_1$, $\mathbf{\Lambda}_1$}
\FOR{$k$=1,2,3,...,$k_{max}$}
\STATE{Solve inequality-constrained LQ GFNE defined by \cref{eq:linear_dyn_iter_k}, \cref{eq:linear_constraint_iter_k}, \cref{eq:quadratic_cost_iter_k}, and denote the solution and corresponding multipliers as $\mathbf{P}_k$, $\bar{\mathbf{\Lambda}}_{k+1}$}
\STATE{Find $\alpha_k$ according to \cref{eq:line_search} and \cref{eq:merit_func} via backtracking line-search}
\IF{$\alpha_k == 0$}
\STATE{Return failure}
\ENDIF
\STATE{$\mathbf{X}_{k+1} \gets \mathbf{X}_k + \alpha_k \mathbf{P}_k$}
\STATE{$\mathbf{\Lambda}_{k+1} \gets \mathbf{\Lambda}_{k} + \alpha_k (\bar{\mathbf{\Lambda}}_{k+1} - \mathbf{\Lambda}_{k})$ }
\IF{$\mathbf{M}(\mathbf{X}_{k+1}, \mathbf{\Lambda}_{k+1}) < \epsilon$}
\STATE{Return $\mathbf{X}_{k+1}, \mathbf{\Lambda}_{k+1}$ and break}
\ENDIF
\ENDFOR
\end{algorithmic}
\end{algorithm}

The choice of merit function need not be as defined in \cref{eq:merit_func}. Any positive-valued function which evaluates to 0 if and only if the arguments constitute a solution to \cref{eq:qfoncs} is acceptable. In the line-search procedure \cref{eq:line_search} using any such merit function, it is necessary in general to evaluate the policy quasi-gradients at the candidate point $\mathbf{X}_k + \alpha \mathbf{P}_k$, $\mathbf{\Lambda}_k + \alpha (\bar{\mathbf{\Lambda}}_{k+1} - \mathbf{\Lambda}_k))$ for general $\alpha$. However, recall that these terms are implicitly defined, and in general require solving the approximate LQ game defined at the candidate point to evaluate them. This makes the evaluation of \cref{eq:line_search} very expensive. Therefore, in practice, we find it acceptable in most cases to replace the policy quasi-gradients appearing in \cref{eq:merit_func} corresponding to the candidate point, with the quasi-gradients corresponding to the point $\mathbf{X}_k, \mathbf{\Lambda}_k$, which are evaluated in the computation of the search direction $\mathbf{P}_k$. Also, even with the reuse of the policy quasi-gradients, the minimization in \cref{eq:line_search} cannot be carried out exactly. In practice, a backtracking line-search satisfying a sufficient decrease condition is used instead. 


\cref{alg:nlgames} gives the full algorithm for computing solutions for nonlinear games. \revise{It is important to note that the proposed algorithm is not guaranteed to converge to a GFQNE solution. Due to the omission of second-order information of the policies $\hat{\pi}_t$ in the definition of the matrices $Q_{t,k}^i$, it is not possible to guarantee that the resultant solution to the LQ game is a descent direction for the proposed merit function. In such a case, failure to converge is reported, as stated in \cref{alg:nlgames}. Despite the possibility of failure, we find in practice that the proposed method is effective at generating search directions which enable progress with respect to \cref{eq:merit_func}.}

\section{Example}
\label{sec:example}

\begin{figure}
\centering
\begin{subfigure}{.5\textwidth}
  \centering
  \includegraphics[width=.9\linewidth]{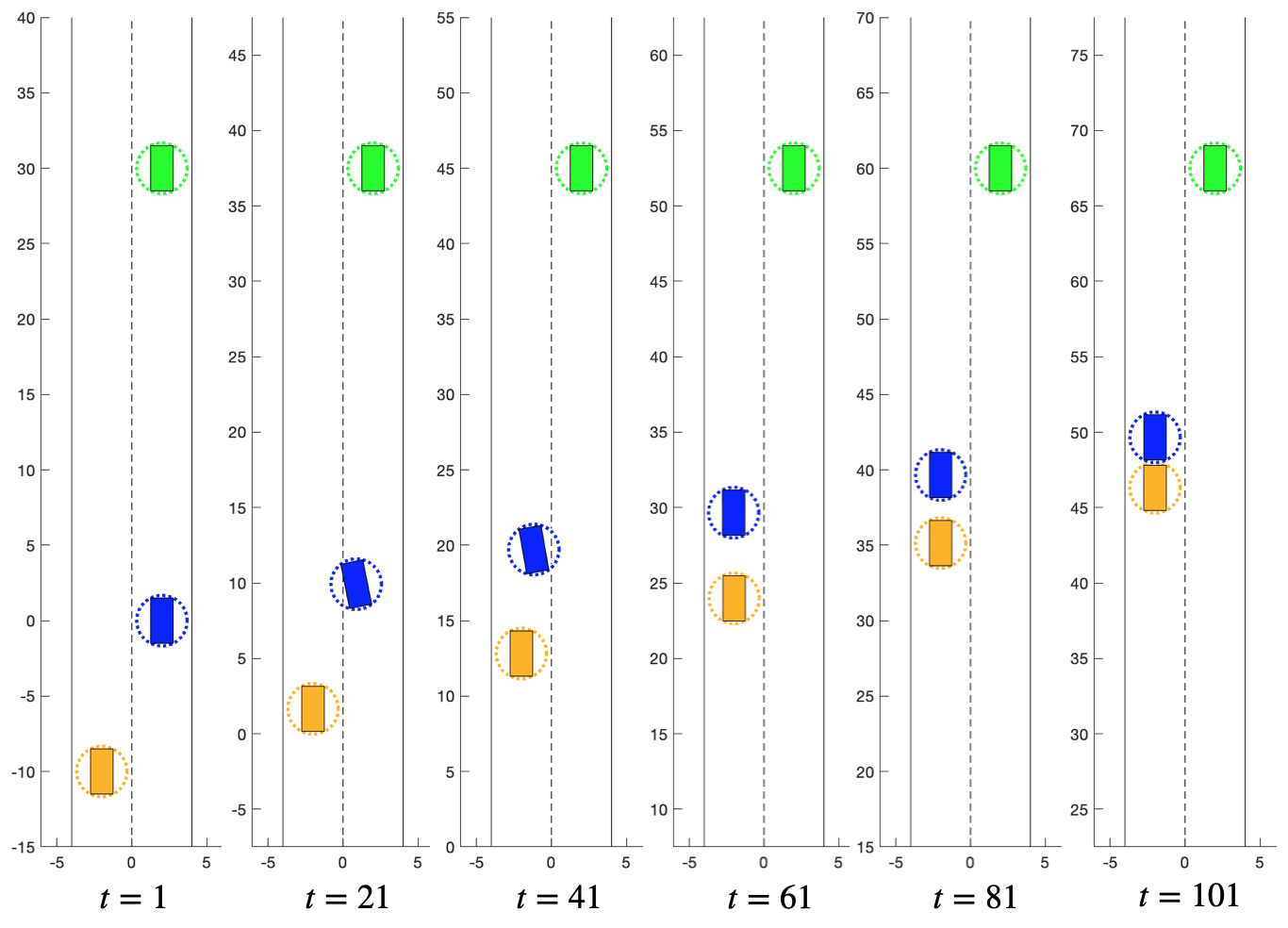}
  \caption{$\sigma_{polite} = 0$}
  \label{fig:sub1}
\end{subfigure}%
\begin{subfigure}{.5\textwidth}
  \centering
  \includegraphics[width=.9\linewidth]{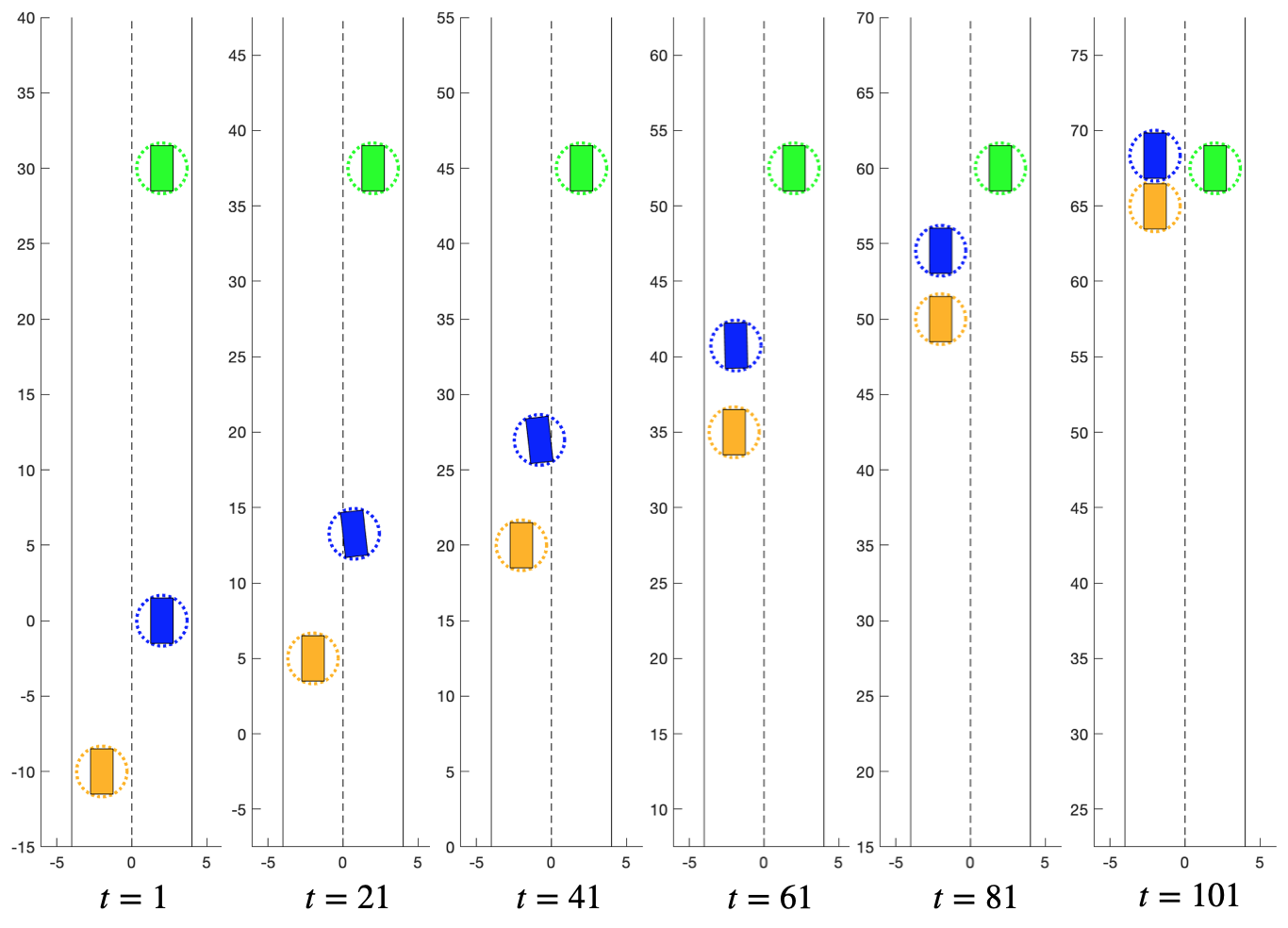}
  \caption{$\sigma_{polite} = 5$}
  \label{fig:sub2}
\end{subfigure}
\caption{Snapshots of the resultant GFQNE solutions found to the example in \cref{sec:example}, when excluding (a) and including (b) the politeness term in player $1$'s objective.} 
\label{fig:driving}
\end{figure}

\begin{table}
    \centering
    \begin{tabular}{cccccc} \toprule
    \makecell{Major\\  Iteration} & \makecell{Minor \\ Iteration }& \makecell{Working Set \\ Indices} & Comment & $\alpha_k$ & $\mathbf{M}(\mathbf{X}_{k+1},\mathbf{\Lambda}_{k+1})$ \\ \midrule
    1  & F1 & \{\} &    \\
       & F2  & \{(101,2)\} &  \\
       & F3  & \{(101,1),(101,2)\} &    \\
       & 1 & \{(101,1),(101,2)\} &   \\
       & 2 & \{(101,2)\} &   \\
       & 3 & \{(101,1),(101,2)\} &  Cycle  & 1 & 96.25 \\ \midrule
    2  & F1 & \{(101,2)\} &     \\
       & F2 & \{(101,1),(101,2)\} &    \\
       & 1 & \{(101,1),(101,2)\} & \\
       & 2  & \{(101,2)\} &    \\
       & 3  &  \{(101,1),(101,2)\} &  Cycle & 1  & 55.65 \\ \midrule
    3  & F1  & \{(101,2)\} &   \\
       & 1   & \{(101,2)\} & Solution  & 1 & 0.221 \\ \midrule
    4  & F1  & \{(101,2)\} &   \\
       & 1   & \{(101,2)\} & Solution  & 1 & 2.3e-5 \\ \bottomrule
    \makecell{Total \\ Time} & \makecell{Total\\ LQ Solves} & Solve Time & \makecell{Function \\ Eval Time}  \\ \toprule 
    5.14 & 17 & 0.97 & 4.17  \\ \midrule
\end{tabular}
    \caption{Algorithm iterate information when using \cref{alg:nlgames} (Major Iterations) and \cref{alg:lq_ineq} (Minor Iterations) to compute a GFQNE to the example in \cref{sec:example}, when $\sigma_{polite} = 5$ (\cref{fig:sub2}).  Here $\mathbf{M}(\mathbf{X}_{k+1},\mathbf{\Lambda}_{k+1})$ is the merit function value after performing a line search in the direction of $\mathbf{P}_k$ in \cref{alg:nlgames}. In each major iteration of this solve, $\alpha_k = 1$, i.e., no backtracking was necessary during line search. 
    Here, minor iterations labeled ``F$\Box$'' indicate equality-constrained LQ solves used to search for a feasible initial solution to the inequality-constrained LQ game associated with the major iteration. Detections of cycles indicate that the removal of a constraint associated with a negative multiplier did not move the iterate associated with \cref{alg:lq_ineq} away from the dropped constraint boundary (\cref{cycle}). In these cases, the iterate is accepted and the algorithm continues with the next major iteration. }
    \label{tab:driving}
\end{table}

We now demonstrate the methodologies so far presented on a practical example.  
Consider a game describing a driving scenario involving an autonomous vehicle and two other vehicles on a freeway. Here $N=3$, and $T=100$ denotes the number of discrete time-points in the trajectory game. Let the game dynamics be defined as the concatenation of the independent dynamics of each vehicle in the game. We assume that each vehicle follows simple unicycle dynamics; specifically:
\begin{equation}
    x_{t+1}^i = \begin{bmatrix*}[l] x_{t+1}^{i,1} \\
    x_{t+1}^{i,2} \\
    x_{t+1}^{i,3}  \\
    x_{t+1}^{i,4} 
    \end{bmatrix*} = f_t^i(x_t^i, u_t^i) = \begin{bmatrix*}[l] x_t^{i,1} + \Delta \cdot x_t^{i,3} \cos(x_t^{i,4}) \\
    x_t^{i,2} + \Delta \cdot x_t^{i,3} \sin(x_t^{i,4}) \\
    x_t^{i,3} + \Delta \cdot u_t^{i,1} \\
    x_t^{i,4} + \Delta \cdot u_t^{i,2}
    \end{bmatrix*}, \hspace{5mm} i \in \mathbf{N}.
\end{equation}
Here $\Delta$ represents some small sampling time. 
The vehicle state dimensions denote, in order, the longitudinal (along-lane) position, lateral (across-lane) position, speed in the direction of the vehicle heading, and the vehicle heading in radians. The control dimensions denote, in order, the vehicle acceleration and the angular velocity.
The dynamics for the entire game state are then given as:
\begin{equation}
    x_{t+1} := \begin{bmatrix} f_t^1(x_t^1,u_t^1) \\  f_t^2(x_t^2,u_t^2) \\  f_t^3(x_t^3,u_t^3) \end{bmatrix} = f_t(x_t,u_t), \ t\in\mathbf{T}.
\end{equation}

Players $2$ and $3$ both aim to minimize acceleration, angular velocity, and deviation from desired speed, while staying in-lane. Player $3$ also desires to avoid collisions with player $1$. 
Player $1$ aims to complete a lateral lane change while minimizing its own acceleration, angular velocity, and deviation from its desired speed, while avoiding collision with player $3$. 
Player 1 also attempts to minimize the objective of player $2$ in addition to its own objective. \cref{fig:driving} depicts this game.
The players' objective and constraints are described as follows:

\begin{equation}
    \begin{aligned}
    &0 = h_{T+1}^1(x_{T+1}) := x^{1,2}_{T+1} + 2,  \\
    &0 = h_{T+1}^2(x_{T+1}) := x^{2,2}_{T+1} + 2,  \\
    &0 = h_{T+1}^3(x_{T+1}) := x^{3,2}_{T+1} - 2,  \\
    &0 \leq g_t^1(x_t) := 
     \left\| \begin{matrix} x_t^{1,1}-x_t^{3,1} \\ x_t^{1,2}-x_t^{3,2}\end{matrix} \right\|_2 - d_{min} , \ \ t\in\mathbf{T^+}, \\
     &0 \leq g_t^2(x_t) := 
     \left\| \begin{matrix} x_t^{2,1}-x_t^{1,1} \\ x_t^{2,2}-x_t^{1,2}\end{matrix} \right\|_2 - d_{min} , \ \ t\in\mathbf{T^+}, 
     \end{aligned}
\end{equation}
\begin{equation}
    \begin{aligned}
    &L^1(x,u) = \Big( \sum_{t=1}^T  \sigma_1 \left\| u_t^1 \right\|_2^2 + \sigma_2 (x_t^{1,2}+2)^2 + \sigma_3 (x_t^{1,3}-v_{goal}^1)^2 \Big) + \sigma_{polite} L^2(x,u), \\
    &L^2(x,u) = \sum_{t=1}^T  \sigma_1 \left\| u_t^2 \right\|_2^2 + \sigma_2 (x_t^{2,2}+2)^2 + \sigma_3 (x_t^{2,3}-v_{goal}^2)^2, \\
    &L^3(x,u) = \sum_{t=1}^T  \sigma_1 \left\| u_t^3 \right\|_2^2 + \sigma_2 (x_t^{3,2}-2)^2 + \sigma_3 (x_t^{3,3}-v_{goal}^3)^2, \\
\end{aligned}
\end{equation}
The initial state is defined to be $\hat{x}_1 := [0, 2, 1, 0, -10, -2, 1.5, 0, 30, 2, 0.75, 0]^\intercal$. The lateral centers of the left and right lanes are $-2$ and $2$, respectively. The constant $d_{min}$ is the minimum separation distance from vehicle centers needed to avoid collision, which in this example is $3.3$. The parameters $\sigma_{\{1,2,3\}}$ scale the relative weights between objective terms, and are $\sigma_1=10,\ \sigma_2=0.2,\text{ and } \sigma_3=10$. The desired speeds of the three players are $v_{goal}^1 = 1$, $v_{goal}^2 = 1.5$, and $v_{goal}^3 = 0.75$. The term $\sigma_{polite}$ is the politeness coefficient and weights how much player $1$ cares about interfering with player $2$'s objective. We consider two variants, with $\sigma_{polite}=0$ and $\sigma_{polite}=5$. This example is similar to games explored in \cite{laine2020multihypothesis}. Visualizations of the GFQNE solutions for the two variants are given in \cref{fig:driving}, and computation details for the $\sigma_{polite}=5$ case are given in \cref{tab:driving} (details for the case $\sigma_{polite} = 0$ are omitted for brevity).

\section{Conclusion}
\label{sec:conclusion}

In this paper, we presented a non-parametric, implicit policy formulation for generalized feedback Nash equilibrium problems. We developed efficient solution methods for the equality-constrained and inequality-constrained Linear-Quadratic (LQ) cases and the general nonlinear case. To the best of our knowledge, these comprise the first solution methods for finding Feedback Nash Equilibria for both LQ and nonlinear games with general constraints. Dynamic games have numerous applications; we demonstrate the utility of our method in a trajectory planning setting for a lane-changing autonomous vehicle.

Future work should consider other solution methods for the general game which build upon our solution to the equality-constrained LQ game. In particular, penalty methods and interior point methods may also be competitive. Furthermore, the results presented should be extended to cases in which strict complementarity does not hold. Necessary conditions based on sub-differentials could be used in such cases. It is also important to further develop a deeper theoretical understanding of the ``policy quasi-gradient'' approximation and its implications on convergence to local solutions. Finally, a high-performance, optimized implementation of our method will facilitate its use in practical applications by other researchers and practitioners.













\bibliographystyle{siamplain}
\bibliography{references}
\end{document}